\documentclass[11pt,reqno]{amsart}
\usepackage{amsthm,amssymb,amsmath}
\usepackage{textcomp}
\usepackage[foot]{amsaddr}
\usepackage{xcolor}
\usepackage{mathrsfs}
\usepackage{hyperref}
\usepackage[T1]{fontenc}
\usepackage{lmodern}
\hypersetup{
	colorlinks =true , 
	linkcolor =blue,
	urlcolor = magenta,
	citecolor =red}

\newtheorem{theorem}{Theorem}[section]
\newtheorem*{theorem*}{Theorem}
\newtheorem{lemma}{Lemma}[section]

\newtheorem{proposition}{Proposition}[section]

\theoremstyle{definition}

\newtheorem{definition}{\bf Definition}[section]
\newtheorem*{definition*}{\bf Definition}

\newtheorem{example}{\bf Example}[section]

\newtheorem{remark}{Remark}
\newtheorem*{remark*}{Remark}

\newtheorem*{example*}{\bf Example}

\newcommand{\spana}{{\mbox{span}}}

\begin{document}

	\title[Linking theorems for multivalued functionals]{Linking theorems for multivalued functionals and application to partial differential inclusions}

	\author{Ablanvi Songo $^{1, a}$}
	\author{Fabrice Colin $^b$}
	\footnote{Corresponding author} 
{\address{ $^a$ Département de mathématiques, Université de Sherbrooke, Sherbrooke, Québec, Canada}
	\email{\href{mailto:ablanvi.songo@usherbrooke.ca}{{\textcolor{blue}{ablanvi.songo@usherbrooke.ca}}}}
	\address{$^b$ School of Engineering and Computer Science, Laurentian University, Sudbury, Ontario, Canada}
	\email{\href{mailto:fcolin@laurentian.ca}{{\textcolor{blue}{fcolin@laurentian.ca}}}}
	\noindent
	

	\begin{abstract}
		Using the framework of critical point theory for multivalued functionals developed in \cite{Fri}, we establish some linking theorems for such functionals, which generalizes \cite[Theorem 2.11]{Wi}, \cite[Theorem 2.12]{Wi} and \cite[Theorem 2.12]{Fri}. The proofs of our abstract results rely on a general minimax principle for multivalued mappings. As an application, we obtain a nontrivial solution of a semilinear Dirichlet boundary value problem for partial differential inclusions.
	\end{abstract}

	\maketitle

	\section{Introduction}
	
In 1997, M.~Frigon developed in \cite{Fri} a critical point theory for
multivalued mappings with closed graph. In that work, Mountain Pass type
and minimax results were obtained. As an application, the author
established an existence result for partial differential inclusions.

In 2001, A.~Kristály and C.~Varga \cite{KV} showed how a quantitative
deformation lemma for continuous functionals can be used to locate
min--max critical points of multivalued functionals whose values have
closed graph. They also obtained Mountain Pass type results for
multivalued functionals.

Later, in 2002, A.~Kristály and C.~Varga \cite{KV2} proved a general
minimax theorem for multivalued mappings. As an application, they
obtained an existence result of critical points for multivalued mappings
satisfying the Cerami condition.

The aim of this paper is to extend the classical Saddle Point and
Linking theorems to multivalued functionals with closed graph. Our
Linking Theorem generalizes the Mountain Pass Theorem for multivalued
functionals established in \cite{Fri} and \cite{KV}.

To the best of our knowledge, this is the very first multivalued version of the classical Saddle and Linking theorems in the literature.

In Section~2, we recall the framework of the critical point theory for
multivalued functionals, together with the classical deformation lemma
for continuous functionals presented in \cite{KV2}.

In Section~3, we establish a quantitative deformation theorem and a
minimax principle for multivalued functionals. Using these results, we
obtain in Section~4 several critical point theorems in the multivalued
context, including a Saddle Point Theorem and a Linking Theorem.

Finally, in the last section we provide an application to a semilinear
Dirichlet boundary value problem for partial differential inclusions.

	\section{Preliminaries}
	Let $(E,d)$ be a metric space endowed with the metric $d$.
	\subsection{Weak slope for continuous functionals}
	  For $r>0$, define
	\[B(u,r) := \Big\{w\in E\;|\; d(u,w)<r\Big\}.\] 
	\begin{definition}[Weak slope, \cite{DM}]
		\label{def 2.1}
		Let $f: E \to \mathbb{R}$ be a continuous function and let $u\in E$. We denote by $|df|(u)$ the supremum of $\sigma \in [0, \infty)$ such that there exist $\delta >0$ and a continuous map 
		\begin{equation*}
			\mathcal{H} : B(u,\delta)\times [0,\delta] \rightarrow E
		\end{equation*}
		such that  for all  $v \in B(u,\delta)$, for all $ t\in [0,\delta]$, we have 
		\begin{enumerate}
			\item[$(a)$] $d(\mathcal{H}(v,t),v) \le t$,
			\item[$(b)$]  $f(\mathcal{H}(v,t))\le f(v) -\sigma t$.
		\end{enumerate}
		The extended real number $|df|(u)$ is called the weak slope of $f$ at $u$.
	\end{definition}
	\begin{definition}
		Let $f: E \to \mathbb{R}$ be a continuous function.
		We say that $u\in E$ is a critical point for $f$ if $|df|(u) = 0$. A real number $c$ is said to be a critical value for $f$, if there exists $u\in E$ such that $|df|(u)=0$ and $f(u)=c$.
		
	\end{definition}
	
	\begin{definition}[\cite{DM}]
		Let $f: E \to \mathbb{R}$ be a continuous function and $c \in \mathbb{R}$. We say that $f$ satisfies the Palais$-$Smale condition at level $c$ ($(PS)_c$ for short), if from every sequence $(u_h)$ in $E$ with $|df|(u_h)\to 0$ and $f(u_h)\to c$ as $h\to \infty$ it is possible to extract a subsequence $(u_{h_j})$ converging in $E$ (the limit of $(u_{h_j})$ is necessarily a critical point for $f$, see \cite[Proposition 2.6]{DM} ).
	\end{definition}
	\subsection{A quantitative deformation theorem for continuous functionals}
	 Let $S \subset E$ and $f : E \rightarrow \mathbb{R}$ be a continuous functional.
	
	We denote by 
	\begin{equation*}
		d(u, S):= \|u-S\|.
	\end{equation*}
	
	For every $\alpha, \beta \in \mathbb{R}$, we denote by
	\[S_{\alpha} :=\Big \{u \in E \;|\; d(u,S) \le \alpha\Big\}, \quad f_{\beta} :=\Big\{u \in E \;|\; f(u)\ge \beta \Big\},\]
	\[f^{\alpha} := \Big\{u \in E \;|\; f(u)\le \alpha \Big\},\quad f_\beta^{\alpha}:=f_\beta \cap f^\alpha. \]
	\begin{theorem}[Deformation theorem, \cite{KV2}]
		\label{theo 2.1}
		Let $(E,d)$ be a complete metric space, $f:E\to \mathbb{R}$ a continuous function, $S$ a closed subset of $E$, $c\in \mathbb{R}$ and $\lambda>0$. Let $\varepsilon>0$ such that
		\begin{equation*}
			\text{for every}\;\; u\in f^{-1}\Big([c-2\varepsilon, c+2\varepsilon]\Big)\cap S_{2\varepsilon}, \;\; \text{we have}\;\; |df|(u)>\dfrac{\varepsilon}{\sqrt{1+\varepsilon^2}}.
		\end{equation*}
		Then, there exists a continuous map
			$\eta: [0,1]\times E \to E$	such that
		\begin{enumerate}
			\item[$(a)$] for every $t \in [0,1] $, \[ d(\eta(t,u),u)\le \lambda t ;\]
			\item[$(b)$] for every $t\in [0,1]$ and for every $u\in E$, \[ f(\eta(t,u))\le f(u); \]
			\item[$(c)$] if $u\notin f^{-1}\Big([c-2\varepsilon, c+2\varepsilon]\Big) \cap S_{2\varepsilon}$, then, for every $t\in[0,1]$,
			\[\eta(t,u)=u;\]
			\item[$(d)$] \[\eta(1, f^{c+\varepsilon'}\cap S)\subset f^{c-\varepsilon'},\] where $\varepsilon' =\dfrac{\varepsilon \min(\varepsilon, \lambda)}{2 \sqrt{1+\varepsilon^2}}$;
			\item[$(e)$] for every $t\in \, ]0,1]$ and for every $u\in f^c \cap S$, we have
			\[f(\eta(t,u))<c.\]
		\end{enumerate}
	\end{theorem}
	\subsection{Weak slope for multivalued functionals}
Let $F : E \rightarrow \mathbb{R}\cup \{\infty\}$ be a multivalued mapping with closed (and non empty values). We denote by
	\begin{equation*}
		\mathbf{\mathit{graph}\,F} = \Big\{(u,c)\in E\times \mathbb{R}\;\Big|\;  c \in F(u)\Big\}.
	\end{equation*}
	On $\mathbf{\mathit{graph}\,F}$, we introduce the metric $d_F$ defined by
	\begin{equation*}
		d_F\Big((u,c),(v,b)\Big) = \sqrt{\Big(d(u,v)\Big)^2 + |b-c|^2}.
	\end{equation*} 
	\begin{remark}
	The set $\mathbf{\mathit{graph}\,F}$ with the metric $d_F$ $\big(\mathbf{\mathit{graph}\,F}, d_F \big)$ is a complete metric space. Indeed, let $(x_n, s_n) \subset E \times \mathbb{R}$ be a Cauchy sequence such that $s_n \in F(x_n)$. Since $E\times \mathbb{R}$ is a compete space, then $(x_n,s_n)$ converges to a limit $(x,s)\in E\times \mathbb{R}$. In addition, since $\mathbf{\mathit{graph}\,F}$ is closed in $E\times\mathbb{R}$, we have $ (x,s)\in \mathbf{\mathit{graph}\,F}$.
	\end{remark}
	Let $B\Big((u,c), \delta\Big)$ be the open ball in $\mathbf{\mathit{graph}\,F}$, centered at $(u,c)$ of radius $\delta>0$.
	\begin{definition}[Weak slope for multivalued mapping, \cite{Fri}]
		
		\label{def 2.4}
		Let $F : E \rightarrow \mathbb{R}\cup \{\infty\}$ be a multivalued mapping with closed graph, and let $(u,c) \in \mathbf{\mathit{graph}\,F}$. The \textit{weak slope} of $F$ at $(u,c)$, denoted by $|dF|(u,c)$ is the supremum of $\sigma \in [0, \infty)$ such that there exist tel $\delta >0$, and a continuous function 
		\begin{equation*}
			\mathcal{H} =(\mathcal{H}_1, \mathcal{H}_2): B\Big((u,c),\delta\Big)\times [0,\delta] \rightarrow \mathbf{\mathit{graph}\,F},
		\end{equation*}
		such that, for every $ (v,b) \in B\Big((u,c),\delta\Big)$ and every $ t\in [0,\delta]$, we have
		\begin{enumerate}
			\item[$(a)$] $d_F\Big(\mathcal{H}\big((v,b),t\big),(v,b)\Big) \le t\sqrt{1 + \sigma^2}$;
			\item[$(b)$] $\mathcal{H}_2\big((v,b),t\big) \le b -\sigma t$.
		\end{enumerate}
	\end{definition}
	\begin{remark}[\cite{Fri}]
		In the case where $F(u) = \{f(u)\}$ is a continuous single-valued function, then \[|dF|(u, f(u))= |df|(u),\] where $|df|(u)$ is the weak slope of $f$ defined in Definition $\ref{def 2.1}$.
	\end{remark}
	Let us consider the function
	\begin{eqnarray*}
		\mathcal{G}_F &:& \mathbf{\mathit{graph}\,F} \to \mathbb{R}\\
		(u,b)	&\mapsto& b,
	\end{eqnarray*}
	where $F$ is as before. Then, $\mathcal{G}$ is continuous.
	
	The following result compares the weak slope of $\mathcal{G}_F$ with that of the multivalued map $F$:
	\begin{proposition}[\cite{Fri, KV2}]
		\label{prop 2.1}
		Let $(u,b)\in \mathbf{\mathit{graph}\,F}$. Then, 
		\begin{eqnarray*}
			|d\mathcal{G}_F|(u,b)&=& \dfrac{|dF|(u,b)}{\sqrt{1+|dF|^2(u,b)}},\quad \text{if}\quad |dF|(u,b)< \infty;\\
			|d\mathcal{G}_F|(u,b)&\ge& 1, \quad \text{if}\quad |dF|(u,b)=  \infty.
		\end{eqnarray*}
	\end{proposition}
In the multivalued context, the notions of critical point of $F$ and of the Palais–Smale condition are well known.
	\begin{definition}[\cite{Fri}, Definition 2.4]
		Let $F : E \rightarrow \mathbb{R}\cup \{\infty\}$ be a multivalued mapping with closed graph, and let $c\in \mathbb{R}$. We say that $u \in X$ is a critical point of $F$ at level $c$, if $c\in F(u)$ and $|dF|(u,c) =0$. Set \[K_c := \Big\{ u \in E \;\Big|\; |dF|(u,c)=0\; \text{and}\; c \in F(u)\Big\}.\] Then, $K_c$ is the set of critical points of $F$ at level $c$. We say that $c$ is a critical value of $F$ if $K_c \ne \emptyset$.
	\end{definition}

	\begin{definition}[\cite{Fri}, Definition 2.5]
		\label{def 2.6}
	Let $F : E \rightarrow \mathbb{R}\cup \{\infty\}$ be a multivalued mapping with closed graph, and let $c\in \mathbb{R}$. The function $F$ satisfies the \textit{Palais$-$Smale condition} at level $c$ $(\text{$(PS)_c$})$ if every sequence $(u_n)$ in $E$ for which there exists $c_n\in F(u_n)$ with $c_n\to c$ and $|dF|(u_n,c_n)\to 0$, has a convergent subsequence in $E$.
	\end{definition}

	\section{A Quantitative deformation theorem and a general minimax principle for multivalued functionals}
	In this section, we establish a quantitative deformation theorem and a general minimax principle for multivalued mapping. 
	
	Let $F: E\rightarrow \mathbb{R}\cup \{\infty\}$ be a multivalued mapping with closed graph, $\delta>0$, and let $U$ be a closed subset of $\mathbf{\mathit{graph}\,F}$. Define
	\begin{equation*}
		U_\delta:=\Big\{(u,b)\in \;\mathbf{\mathit{graph}\,F}\; \;\Big|\; d_F\Big((u,b), U\Big)\le \delta\Big\}.
	\end{equation*}
	
		\subsection{A quantitative deformation theorem}
	We now propose a deformation theorem for multivalued mapping.
	\begin{theorem}[Deformation theorem for multivalued mapping]
		\label{theo 3.1}
	Let $F: E\rightarrow \mathbb{R}\cup \{\infty\}$ be a multivalued mapping with closed graph, $c\in \mathbb{R}$, and let $U$ be a closed subset of $\mathbf{\mathit{graph}\,F}$. Let $\lambda>0$. Let $\varepsilon_0 >0$ and $\delta>0$ such that
		\[\text{for every}\;\;(u,b)\in \Big\{(u,b)\in\; E\times \mathbb{R} \;\Big|\; b\in  [c-2\varepsilon_0, c+2\varepsilon_0] \Big\} \cap U_{2\delta},\;\;\text{we have}\]
		\begin{equation}
			\label{eq 1}
			|dF|(u,b)>\varepsilon_0.
		\end{equation}
	There exists a continuous function 
			$\eta =(\eta_1,\eta_2): [0,1]\times \mathbf{\mathit{graph}\,F} \to \mathbf{\mathit{graph}\,F}$ such that
		\begin{enumerate}
			\item[$(i)$] for every $t\in [0,1]$, \[d_F\Big(\eta(t,(u,b)),(u,b)\Big)\le \lambda t;\]
			\item [$(ii)$] \[\eta_2(t,(u,b))\le b;\]
			\item [$(iii)$] if $(u,b) \notin  \Big\{(u,b)\in\; E\times \mathbb{R} \;\Big|\; b\in  [c-2\varepsilon_0, c+2\varepsilon_0] \Big\} \cap U_{2\delta}$, then \[\eta(t,(u,b))=(u,b);\]
			\item [$(iv)$] \[\eta\Big(1,U\cap E\times (-\infty, c+\varepsilon']\Big) \subset E\times (-\infty, c-\varepsilon'], \;\; \text{for some}\;\; \varepsilon' \in (0,\varepsilon_0); \] 
			\item [$(v)$] for every $t\in ]0,1]$ and for every $(u,b)\in \Big(E\times(-\infty, c]\Big)\cap U$, we have \[\eta_2(t,(u,b))<c.\]
		\end{enumerate}
	\end{theorem}
	\begin{proof}
		Let us consider the continuous single-valued function
		\begin{eqnarray*}
			\mathcal{G}_F : \mathbf{\mathit{graph}\,F} &\to& \mathbb{R}\\
			(u,b)	&\mapsto& b.
		\end{eqnarray*}
		By Proposition $\ref{prop 2.1}$ and $(\ref{eq 1})$, for every $(u,b)\in U_{2\delta}\cap \mathcal{G}^{-1}_F\Big([c-2\varepsilon_0, c+2\varepsilon_0]\Big)$, we have	\[|d\mathcal{G}_F|(u,b)= \dfrac{|dF|(u,b)}{\sqrt{1+|dF|^2(u,b)}}> \dfrac{\varepsilon_0}{\sqrt{1+\varepsilon_0^2}},\]
		because, the function \begin{equation*}
			x\mapsto \dfrac{x}{\sqrt{1+x^2}}
		\end{equation*}
		is increasing, for every $x\in \mathbb{R}$.\\ We apply the deformation theorem (Theorem $\ref{theo 2.1}$) with \[\varepsilon =\min\{\varepsilon_0, \delta\}.\] For a fixed $\lambda > 0$, we obtain the existence of
		  \[\varepsilon' =\dfrac{\varepsilon_0 \min(\varepsilon_0, \lambda)}{2 \sqrt{1+\varepsilon_0^2}} \in (0,\varepsilon_0),\] and a continuous function $\eta =(\eta_1,\eta_2): [0,1]\times \mathbf{\mathit{graph}\,F}  \to \mathbf{\mathit{graph}\,F}$ such that\\
		$(1)$ \[d_F\Big(\eta(t,(u,b)),(u,b)\Big)\le \lambda t;\]
		$(2)$ \[\eta_2(t,(u,b)) = \mathcal{G}_F \Big(\eta_1\big(t,(u,b)\big), \eta_2\big(t,(u,b)\big)\Big) \le \mathcal{G}_F (u,b)=b;\]
		$(3)$  if $(u,b)\notin U_{2\delta}\cap \mathcal{G}^{-1}_F\Big([c-2\varepsilon, c+2\varepsilon]\Big)$, then \[\eta(t,(u,b)) = (u,b);\]
		$(4)$ \[\eta\Big(1, U \cap X\times (-\infty, c+\varepsilon']\Big)\subset X\times (-\infty, c-\varepsilon'] ;\]
		$(5)$ for every $t\in ]0,1]$ and for every $(u,b)\in \Big(X\times(-\infty, c]\Big)\cap U$, we have \[\eta_2(t,(u,b)) = \mathcal{G}_F \Big(\eta_1\big(t,(u,b)\big), \eta_2\big(t,(u,b)\big)\Big)<c.\]
	\end{proof}

	\subsection{A general minimax principle}
Let $(\Psi_m)_{m\in \mathbb{N}}$ be a family of homeomorphisms, where \[\Psi_m : \mathbb{R}^m \to\Psi_m(\mathbb{R}^m)\subset \mathbf{\mathit{graph}\,F}.\]
Let us introduce the following notations. Let $m\in \mathbb{N}$.
	\begin{align}
	\label{eq 2}	D_m &:= \Psi_m(B_1^m),\\
	\label{eq 3}	S_m &:= \partial D_m= \Psi_m(\partial B_1^{m}),
	\end{align}
where $B_1^m$ is the unit closed ball centered at 0 in $\mathbb{R}^m$ and $\partial B_1^{m}$, its boundary,
	\[ \Gamma_0^m \subseteq \mathcal{C}(S_m, \mathbf{\mathit{graph}\,F}),\]
	\[ 	\overline{\Gamma}_m:= \Big\{\gamma=\Big(\gamma_1, \gamma_2\Big) \in \mathcal{C}\Big(D_m, \mathbf{\mathit{graph}\,F}\Big)\;\Big|\; \gamma_{|_{S_m}} \in \Gamma_0^m \;\: \text{and}\;\; \gamma_2(u,b)\le b \Big\}.\]

	We denote $\pi_{\mathbb{R}}$ the projection of $\mathbf{\mathit{graph}\,F}$ on $\mathbb{R}$.
	
	\begin{theorem}[General minimax principle for multivalued mapping]
		\label{theo 3.2}
		Let $E$ be a Banach space, and $F : E \rightarrow \mathbb{R}\cup \{\infty\}$ a multivalued mapping with closed graph.
		Assume that
		\begin{equation}
			\label{eq 4}
			\infty> c:= \underset{\gamma \in \overline{\Gamma}_m}{\inf}\;\sup\; \pi_{\mathbb{R}} (\gamma(D_m)) > a:= \underset{\gamma_0 \in \Gamma_0^m}{\sup}\;\sup\; \pi_{\mathbb{R}} (\gamma_0(S_m)).
		\end{equation}
		Then, for every $\varepsilon \in  \Big(0, \frac{c-a}{2}\Big)$, $\delta >0$ and $\gamma \in \overline{\Gamma}_m$ such that
		\begin{equation}
			\label{eq 5}
			{\sup}\; \pi_{\mathbb{R}}(\gamma(D_m))\le c +\varepsilon', \;\: \text{for some}\;\; \varepsilon' \in (0,\varepsilon),
		\end{equation}
		there exists $(u,b) \in \mathbf{\mathit{graph}\,F}$ such that
		\begin{enumerate}
			\item[$(a)$] \[c-2\varepsilon \le b \le c+2\varepsilon,\]
			\item[$(b)$] \[d_F\Big((u,b), \gamma(D_m)\Big) \le 2 \delta,\]
			\item[$(c)$] \[|dF|(u,b)\le \varepsilon.\]
		\end{enumerate}
	\end{theorem}
	\begin{proof}
	Suppose that the conclusion of the theorem is false, that is, there exist $\varepsilon \in \Big( 0, \frac{c-a}{2}\Big)$, $\delta >0$ and $\gamma \in \overline{\Gamma}_m$ satisfying\;
			${\sup}\; \pi_{\mathbb{R}}(\gamma(D_m))\le c +\varepsilon'$ for some $\varepsilon' \in (0,\varepsilon)$, such that
		\[\text{for every}\;\; (u,b)\in \Big(\gamma(D_m)\Big)_{2\delta}\cap \pi_{\mathbb{R}}^{-1}\Big([c-2\varepsilon, c+2\varepsilon]\Big),\;\; \text{we have}\;\; |dF|(u,b) > \varepsilon.\]
		
		For a fixed $\lambda>0$, we apply Theorem $\ref{theo 3.1}$ with  \[U=\gamma(D_m)\;\; \text{and}\;\; \varepsilon_0= \varepsilon.\]
		 Since $\varepsilon \in \Big( 0, \frac{c-a}{2}\Big)$, then 
		\begin{equation}
			\label{eq 6}
			c-2\varepsilon_0\ge	c-2\varepsilon >a.
		\end{equation}
		
		Define the continuous function
		\begin{eqnarray*}
			\beta=\Big(\beta_1, \beta_2\Big) &:& D_m\to \mathbf{\mathit{graph}\,F}\\
			(u,b) &\mapsto& \eta(1, \gamma(u,b)),
		\end{eqnarray*}
		where $\eta$ is given by Theorem $\ref{theo 3.1}$. For every $(u,b)\in S_m$, we have \[\beta(u,b)=\eta(1, \gamma_0(u,b)).\] By $(\ref{eq 6})$, \[\pi_{\mathbb{R}}\Big(\gamma_0(u,b)\Big)<a < c-2\varepsilon_0.\] It follows that \[\gamma_0(u,b)\notin \pi_{\mathbb{R}}^{-1}\Big([c-2\varepsilon_0, c+2\varepsilon_0]\Big).\] By $(iii)$ and $(ii)$ of Theorem $\ref{theo 3.1}$, on $S_m$, we have
		\[\beta(u,b)= \eta(1, \gamma_0(u,b)) =\gamma_0(u,b) \in \Gamma_0^m;\] \[\beta_2(u,b)=\eta_2(1,\gamma(u,b))\le b.\]
		We deduce that $\beta\in \overline{\Gamma}_m$. \\ For every $(u,b) \in D_m$, since \[\pi_{\mathbb{R}} \Big(\gamma(u,b)\Big) \le c+\varepsilon',\;\; \text{and}\;\; \gamma(u,b)\in U,\] then, $(iv)$ of Theorem $\ref{theo 3.1}$ implies that \[\beta(u,b)=\eta(1,\gamma(u,b))\in E\times (-\infty, c-\varepsilon'].\] Hence,
		\begin{equation*}
			{\sup}\; \pi_{\mathbb{R}}\Big(\beta(u,b)\Big) = {\sup}\; \pi_{\mathbb{R}}\Big( \eta(1,\gamma(u,b))\Big) \le c -\varepsilon'.
		\end{equation*}
		Finally, we obtain that
		\begin{equation*}
			c\le {\sup}\; \pi_{\mathbb{R}}\Big( \beta(u,b)\Big) \le c- \varepsilon',
		\end{equation*}
		which is a contradiction.
	\end{proof}
	\begin{theorem}
		\label{theo 3.3}
	Suppose that the assumptions of Theorem $\ref{theo 3.1}$ are satisfied and suppose that $F$ satisfies $(\ref{eq 4})$. Then, there exists a sequence $(u_n,c_n)\subset \mathbf{\mathit{graph}\,F}$ satisfying
	\begin{equation}
		\label{eq 7}
		c_n\to c, \;\; |dF|(u_n,c_n)\to 0,\;\; n\to \infty.
	\end{equation}
	In particular, if $F$ satisfies $(PS)_c$ condition $(\text{see Definition $\ref{def 2.6}$} )$, then $c$ is a critical value of $F$.	
	\end{theorem}
	\begin{proof}
		Under assumptions of Theorem $\ref{theo 3.1}$, suppose, in contrary, that there is no  sequence in $\mathbf{\mathit{graph}\,F}$  satisfying $(\ref{eq 7})$. Denote this assertion by $(T_1)$. We will establish that assertion $(T_1)$ implies the assertion $(T_2)$: there exist $\varepsilon \in \Big( 0, \frac{c-a}{2}\Big),\; \delta >0$ and $\gamma \in \overline{\Gamma}_m$ satisfying $(\ref{eq 4})$ such that \[\text{for every}\;\; (u,b)\in \Big(\gamma(D_m)\Big)_{2\delta}\cap \pi_{\mathbb{R}}^{-1}\Big([c-2\varepsilon, c+2\varepsilon]\Big)\subset \mathbf{\mathit{graph}\,F,} \;\; \text{we have}\;\; |dF|(u,b) > \varepsilon.\]
		To do so, we prove that $\neg (T_2)$ implies $\neg (T_1)$.\\
		Suppose that, for every $\varepsilon$ (take $\varepsilon=\frac{1}{n}$), $\delta >0$ and $\gamma \in \overline{\Gamma}_m$ satisfying $(\ref{eq 4})$, there exists a sequence $(u_n,c_n)\in \Big(\gamma(D_m)\Big)_{2\delta}\cap \pi_{\mathbb{R}}^{-1}\Big(\Big[c-\frac{2}{n}, c+\frac{2}{n}\Big]\Big)\subset \mathbf{\mathit{graph}\,F}$, such that $|dF|(u_n,c_n) \le \frac{1}{n}.$\\
		Since $\pi_{\mathbb{R}}(u_n,c_n)=c_n  \in \Big[c-\frac{2}{n}, c+\frac{2}{n}\Big]$, then
		\[	c_n\to c, \;\; |dF|(u_n,c_n)\to 0,\quad \text{whenever}\;\; n\to \infty.\]
		Thus we have shown that assertion $\neg (T_2)$ implies $\neg (T_1)$.\\
		But then, under assertion $(T_2)$, and for a fixed $\lambda>0$, we can apply Theorem $\ref{theo 3.1}$ with
		\[U=\gamma(D_m)\quad \text{and}\quad \varepsilon_0=\varepsilon.\] Let us define, as  in the proof of Theorem $\ref{theo 3.2}$, the continuous map
		\begin{eqnarray*}
			\beta=\Big(\beta_1, \beta_2\Big) : D_m &\to& \mathbf{\mathit{graph}\,F}\\
			(u,b) &\mapsto& \eta(1, \gamma(u,b)),
		\end{eqnarray*}
		where $\eta$ is given by Theorem $\ref{theo 3.1}$. We obtain a contradiction : 
		\[	c\le {\sup}\; \pi_{\mathbb{R}}\Big( \beta(u,b)\Big) \le c- \varepsilon'.\]
		The proof of Theorem $\ref{theo 3.3}$ is thus complete.
	\end{proof}
		\section{Critical point theorems for multivalued functionals}
	In this section, we give two examples where condition $(\ref{eq 4})$ is satisfied.
	\begin{theorem}[Saddle Point Theorem for multivalued mapping]
		\label{theo 4.1}
	Let $ \Big(E = Y \oplus Z, \|\cdot\|\Big)$ be an infinite-dimensional Banach space, where $Y $ is finite-dimensional. Let $F : E \rightarrow \mathbb{R}\cup \{\infty\}$ be a multivalued mapping with closed graph. Suppose that $F$ satisfies the following condition: 
	\begin{enumerate}
		\item[$(I_1)$] $\underset{u\in Z}{\inf}\; F(u) >0$,
	\item[$(I_2)$] There exist  $\rho>0$ and a homeomorphism $\Psi_m$ such that, for $m=\dim Y$, \[\Psi_m : \mathbb{R}^m \to \Psi_m(\mathbb{R}^m)\subset \mathbf{\mathit{graph}\,F}\;\; \text{and}\;\; \Psi_m(\partial B_1^{m})\subseteq \partial B_\rho \times (-\infty, 0],\] where $B_1^m$ is the unit closed ball centered at $0$ in $\mathbb{R}^m$ with boundary $\partial B_1^m$ and \\$B_\rho =\{ u\in Y\;|\; \|u\|\le \rho\}$.
\end{enumerate}
	 Finally, let $c\in \mathbb{R}$ be characterized by
		\begin{equation*}
			c:= \underset{\gamma \in \Gamma_m}{\inf}\,\max \,\pi_{\mathbb{R}} \Big(\gamma\big(\Psi_m(B_1^m)\big)\Big),
		\end{equation*}
		where
		\begin{equation*}
			\Gamma_m:= \Big\{\gamma=\Big(\gamma_1, \gamma_2\Big) \in \mathcal{C}\Big(\Psi_m(B_1^m), \mathbf{\mathit{graph}\,F}\Big)\;\Big|\; \gamma_{|_{\partial (\Psi_m(B_1^m))}} =\operatorname{id} \;\; \text{and}\;\; \gamma_2(u,b)\le b \Big\}.
		\end{equation*}
	Then, there exists a sequence $(u_n,c_n)\subset \mathbf{\mathit{graph}\,F}$ satisfying
		\begin{equation}
			\label{eq 8}
			c_n\to c, \;\; |dF|(u_n,c_n)\to 0,\;\; n\to \infty.
		\end{equation}
		Such a sequence is called a Palais--Smale sequence at level $c$, or a $(PS)_c$ sequence.
	\end{theorem}
	
	\begin{proof}
		 Denote $ m:=\dim Y $. Let $\rho_0>0$. 
		We consider, for $b\le0$, the map
		\begin{eqnarray*}
			\Psi_m : \mathbb{R}^m  &\to& \Psi_m(\mathbb{R}^m) \subset \mathbf{\mathit{graph}\,F}\\
			x=(x_1,\ldots,x_m)
			&\mapsto& \left( \rho_0 \sum_{i=1}^m x_i e_i, \;b\right).
		\end{eqnarray*}
		Then, $\Psi_m$ is a homeomorphism that satisfies the condition $(I_2)$ (with $\rho= \rho_0\|\sum_{i=1}^m x_i e_i\|$).\\
		We recall the following notations: 
			\begin{equation*}
				D_m := \Psi_m(B_1^m),\;\;
			S_m := \partial D_m = \Psi_m(\partial B_1^{m}).
		\end{equation*}
		
		In order to apply Theorem $\ref{theo 3.3}$, we have only to verify that \[c\ge \underset{u\in Z}{\inf}\; F(u).\]
		Let us prove that, for every $\gamma \in \Gamma_m$, we have $\gamma(D_m)\cap (Z\times \mathbb{R})\ne \emptyset$.\\
		Denote by $\pi_Y$ the projection of $\mathbf{\mathit{graph}\,F}$ onto $Y$ such that $\pi_Y\Big(\mathbf{\mathit{graph}\,F}\cap (Z\times \mathbb{R})\Big)=\Big\{0\Big\}$. Suppose there exists $\gamma \in \Gamma_m$ such that  \[\gamma(D_m)\cap (Z\times \mathbb{R})= \emptyset\] and define the map
	\begin{eqnarray*}
		R_m : B^m_1 &\to& \partial B_\rho\\
		x&\mapsto& \dfrac{\rho}{\Big\|\pi_Y\Big(\gamma\big(\Psi_m(x)\big)\Big)\Big\|}\pi_Y\Big(\gamma\big(\Psi_m(x)\big)\Big).
	\end{eqnarray*}
		The map $R_m$ is well defined and is continuous. \\
		Set
		$L(x_1,\ldots,x_m)
		:= \sum_{i=1}^m x_i e_i$.
		Let us consider the homeomorphism
		\begin{eqnarray*}
			\Phi_Y: \partial B^m_1 &\to& \partial B_\rho\\
			x=(x_1, \ldots, x_m) &\mapsto& \dfrac{\rho}{\|L(x)\|}L(x).
		\end{eqnarray*}
		We claim that the continuous map
		\begin{eqnarray*}
			R: B_1^m &\to&\partial B_1^m\\
			x &\mapsto& \Phi_Y^{-1}\circ R_m(x)
		\end{eqnarray*}
		is a retraction from $B^m_1$ onto $\partial B_1^m$. Indeed, let $x\in \partial B_1^m$. Then, \[R_m(x)= \pi_Y(\Psi_m(x))=\Phi_Y(x)\in \partial B_\rho.\]
		We deduce that
		\[R(x)=\Phi_Y^{-1}(\Phi_Y(x))=x.\]
		It follows that, $R_m$ is a retraction from $B^m_1$ onto its boundary $\partial B^m_1$. Since $\dim Y <\infty$, this is impossible by the Non retractability Theorem  (see \cite[Theorem D.11]{Wi}). Hence, for every $\gamma \in \Gamma_m$, we have \[\gamma(D_m)\cap (Z\times \mathbb{R})\ne \emptyset .\]
		By assumption $(I_2)$, 
		\begin{equation*}
			a= {\max}\; \pi_{\mathbb{R}} (S_m)\le 0 < \underset{u\in Z}{\inf}\; F(u) .
		\end{equation*}
		For every $\gamma \in \Gamma_m$, there exists $\mathcal{E} = \gamma(D_m) $ such that
		\begin{equation*}
			{\max}\; \pi_{\mathbb{R}} (\mathcal{E})\ge \max\; \pi_{\mathbb{R}}((Z\times \mathbb{R})\cap \mathcal{E})\ge \underset{u\in Z}{\inf}\; F(u).
		\end{equation*}
		Therefore,
		\begin{equation*}
			c= \underset{\gamma \in \Gamma_m}{\inf}\; {\max}\; \pi_{\mathbb{R}} (\mathcal{E}) \ge \underset{u\in Z}{\inf}\; F(u).
		\end{equation*}
		Since \[c\ge a= {\max}\; \pi_{\mathbb{R}} (S_m),\] then $F$ satisfies the relation $(\ref{eq 4})$ of Theorem $\ref{theo 3.2}$ with \[\Gamma_0^m =\Big\{\operatorname{id}\Big\} \;\; \text{and}\;\;  \operatorname{id}(S_m)= S_m.\] By applying Theorem $\ref{theo 3.3}$, we obtain the existence of a sequence $(u_n,c_n)\subset \mathbf{\mathit{graph}\,F}$ satisfying $(\ref{eq 8})$.
	\end{proof}

	\begin{theorem}[Linking Theorem for multivalued mapping]
	\label{theo 4.2}
		Let $ \Big(E = Y \oplus Z, \|\cdot\|\Big)$ be an infinite-dimensional Banach space, where $Y $ is finite-dimensional. Let $F : E \rightarrow \mathbb{R}\cup \{\infty\}$ be a multivalued mapping with closed graph. Let $\rho>r>0$ and let $z\in Z$ such that $\|z\|=r$.
		Define
		\begin{equation*}
			M : = \Bigl\{ u=y+\lambda z\; \Big|\; \| u\| \le \rho,\; \lambda \ge 0, \;y \in Y\Bigr\}, \\
		\end{equation*}
		\begin{equation*}
			\partial M := \Bigl\{ u= y + \lambda z\in M \;\Big|\; \Big(\|u\| =\rho \;\; \text{and}\;\; \lambda \ge 0\Big) \; \text{or}\; \Big(\|u\| \le \rho \;\; \text{and}\;\;  \lambda =0\Big)  \Bigr\},\\
		\end{equation*}
		\begin{equation*}
			N := \Bigl\{ u \in Z \;\Big|\; \|u\| =r \Bigr\}.
		\end{equation*}Suppose that $F$ satisfies the following condition: 
		\begin{enumerate}
	\item[$(I_1)$]	$\underset{u\in N}{\inf}\; F(u) >0$,
	\item[$(I_2)$] there exists a homeomorphism $\Psi_m$ such that, for $m=\dim Y +1$, \[\Psi_m : \mathbb{R}^m \to \Psi_m(\mathbb{R}^m)\subset \mathbf{\mathit{graph}\,F}\;\; \text{and}\;\; \Psi_m(\partial B_1^{m})\subseteq \partial M \times (-\infty, 0],\] where $B_1^m$ is the unit closed ball centered at $0$ in $\mathbb{R}^m$ with boundary $\partial B_1^m$. 
\end{enumerate}
	Finally, let $c\in \mathbb{R}$ be characterized by
	\begin{equation*}
		c:= \underset{\gamma \in \Gamma_m}{\inf}\,\max \,\pi_{\mathbb{R}} \Big(\gamma(\Psi_m(B_1^m)\Big),
	\end{equation*}
	where
	\begin{equation*}
		\Gamma_m:= \Big\{\gamma=\Big(\gamma_1, \gamma_2\Big) \in \mathcal{C}\Big(\Psi_m(B_1^m), \mathbf{\mathit{graph}\,F}\Big)\;\Big|\; \gamma_{|_{\partial (\Psi_m(B_1^m))}} =\operatorname{id} \;\; \text{and}\;\; \gamma_2(u,b)\le b \Big\}.
	\end{equation*}
	Then, there exists a sequence $(u_n,c_n)\subset \mathbf{\mathit{graph}\,F}$ satisfying
	\begin{equation}
		\label{eq 9}
		c_n\to c, \;\; |dF|(u_n,c_n)\to 0,\;\; n\to \infty.
	\end{equation}
	Such a sequence is called a Palais--Smale sequence at level $c$, or a $(PS)_c$ sequence.
	\end{theorem}
	
	\begin{proof} Let $z_0 \in Z\setminus \{0\}$. Set $z:=r\frac{z_0}{\|z_0\|}$. Denote $m:= d +1$, where $d:=\dim Y$. Let $\rho_0>0$.	
		We consider, for $b\le 0$, the continuous map
		\begin{eqnarray*}
			\Psi_m : \mathbb{R}^m  &\to& \Psi_m(\mathbb{R}^m) \subset \mathbf{\mathit{graph}\,F}\\
			x=(x_1,\ldots,x_m)
			&\mapsto& \begin{cases}
				\left(\rho_0\big(\sum_{i=1}^d x_i e_i + x_m z\big), \;b\right), &\text{if}\;\; x_m\ge0,\\[0.5 em]
				\left(\rho_0\sum_{i=1}^d x_i e_i, \;b\right), &\text{if}\;\; x_m\le0.
			\end{cases}
		\end{eqnarray*}
		Then, $\Psi_m$ is a homeomorphism that satisfies the condition $(I_2)$ (with $\rho= \rho_0\|L(x)\|$).\\
		 We recall the following notations: 
		\begin{equation*}
			D_m = \Psi_m(B_1^{m}),\;\;
			S_m = \partial D_m= \Psi_m(\partial B_1^{m}).
		\end{equation*}
		
		In order to apply Theorem $\ref{theo 3.3}$, we have only to verify that \[c\ge \underset{u\in N}{\inf}\; F(u).\] Let us prove that, for every $\gamma \in \Gamma_m$, we have $\gamma(D_m)\cap (N\times \mathbb{R})\ne \emptyset$.
		 Denote by $P$ the projection of $\mathbf{\mathit{graph}\,F}$ onto $Y$ such that $P\Big(\mathbf{\mathit{graph}\,F} \cap (Z\times \mathbb{R})\Big)=\Big\{0\Big\}$, $\pi_E$ the projection of $\mathbf{\mathit{graph}\,F}$ onto $E$ and by \[R: (Y\oplus\mathbb{R}z)\setminus\{0+z\}\to \partial M\] a retraction. \\ Suppose there exists $\gamma \in \Gamma_m$ such that \[\gamma(D_m)\cap (N\times \mathbb{R})= \emptyset.\] Then, for every $(u,b)\in D_m$, 
		\begin{equation}
			\label{eq 10}
			\gamma(u,b)\notin N\times \mathbb{R}.
		\end{equation}
		Define the map 
		\begin{eqnarray*}
			\mathcal{R}_m : B_1^m &\to& Y\oplus \mathbb{R}z\setminus\{0+z\} \\
			x&\mapsto& P(\gamma(\Psi_m(x)))+ \Big\| \pi_E(\gamma(\Psi_m(x)))-P(\gamma(\Psi_m(x)))\Big\|r^{-1}z.
		\end{eqnarray*}
		The map $\mathcal{R}_m$ is well defined and continuous. Indeed, let $x\in B_1^m$. Then, there exists $(u,b)\in D_m$ such that $\Psi_m(x)=(u,b)$. Therefore,
		$P(\gamma(u,b))+ \Big\| \pi_E(\gamma(u,b))-P(\gamma(u,b))\Big\|r^{-1}z \in Y\oplus \mathbb{R}z\setminus\{0+z\}$.
		If not, suppose that $P(\gamma(u,b))=0$, and that \[\Big\| \pi_E(\gamma(u,b))-P(\gamma(u,b))\Big\|r^{-1}z=z.\] This implies that
		\[\gamma(u,b)\in \mathbf{\mathit{graph}\,F}\cap (Z\times \mathbb{R})\;\; \text{and}\;\;  \Big\|\pi_E(\gamma(u,b))\Big\| =r.\]
		In other words, \[ \pi_E(\gamma(u,b)) \in Z\;\; \text{and}\;\;	\Big\|\pi_E\Big(\gamma(u,b)\Big)\Big\|= r.\]
		It follows that \[\gamma(u,b)\in N\times \mathbb{R},\]	contradicting $(\ref{eq 10})$.
		
		Let us now consider the continuous map 
		\begin{eqnarray*}
			\Phi_{Y\oplus\mathbb{R}z} : \partial B_1^m &\to& \partial M\\
			x=(x_1,\ldots, x_d, x_m) &\mapsto&
			\begin{cases}
				\dfrac{\rho}{\|L(x)\|}L(x), &\text{if}\;\; x_m\ge0,\\[1 em]	
				\dfrac{\rho}{\|L(x)\|}P(L(x)), &\text{if}\;\; x_m\le0.
			\end{cases} 
		\end{eqnarray*}
		Then, $\Phi_{Y\oplus\mathbb{R}z}$ is a homeomorphism. We claim that the map
		\begin{eqnarray*}
			\mathcal{A}: B_1^m &\to&\partial B_1^m\\
			x &\mapsto& (\Phi_{Y\oplus\mathbb{R}z})^{-1}\circ R\circ \mathcal{R}_m(x)
		\end{eqnarray*}
		is a retraction from $B^m_1$ onto $\partial B^m_1$. \\
		Indeed, $\mathcal{A}$ is continuous. Let $x\in \partial B_1^m$. Then, there is $(u,b)\in \Psi_m(\partial B_1^m)$ such that $\Psi_m(x)=(u,b)$. Furthermore, $u\in \partial M$ (by $(I_2)$). We have  \begin{equation*}
			\mathcal{R}_m(x)= P(\Psi_m(x))+\lambda z =\begin{cases}
				\rho_0\sum_{i=1}^d x_i e_i + \lambda z = \rho_0\Big(\sum_{i=1}^d x_i e_i + \frac{\lambda}{\rho_0} z\Big), &\text{if}\;\; \lambda> 0,\\[0.5 em]
				\rho_0 \sum_{i=1}^d x_i e_i, &\text{if}\;\; \lambda=0.
			\end{cases}
		\end{equation*}
		From condition $(I_2): \Psi(\partial B_1^{m}) \subseteq M_0 \times (-\infty, 0]$ $\Big($with $\rho=\rho_0\|L(x)\| \Big)$,\\
		
		\begin{enumerate}
			\item [$(i)$] On $\partial B_1^m\cap \{x_m\ge0\}$, si $\lambda>0$, one would get
			\[\mathcal{R}_m(x)=\Phi_{Y\oplus\mathbb{R}z}(x)\cap\{x_m:=\lambda/\rho_0\ge0\}. \] 
			\item[$(ii)$] On $\partial B_1^m\cap \{x_m\le0\}$, $\|P(\Psi_m(x))\|=\rho_0\|P(L(x))\|\le \rho$. So, if $\lambda=0$, then
			\[\mathcal{R}_m(x)= P(\Psi_m(x))= \rho_0P(L(x))\cap\{x_m\le0\}=\Phi_{Y\oplus\mathbb{R}z}(x)\cap\{x_m\le0\}. \]
		\end{enumerate}
		Therefore, on $\partial B_1^m$, $\mathcal{R}_m(x)=\Phi_{Y\oplus\mathbb{R}z}(x)$. It follows that $R(\mathcal{R}_m(x))= \mathcal{R}_m(x)$.
		Hence, 
		\[\mathcal{A}(x)= (\Phi_{Y\oplus\mathbb{R}z})^{-1} (\mathcal{R}_m(x))= (\Phi_{Y\oplus\mathbb{R}z})^{-1} (\Phi_{Y\oplus\mathbb{R}z} (x))=x\;\; \text{on}\;\; \partial B_1^m.\]
		The map $\mathcal{A}$ is a retraction from $B^m_1$ onto its boundary $\partial B^m_1$. This is impossible by the Non retractability Theorem since $ \dim Y <\infty$ (see \cite[Theorem D.11]{Wi}). Hence, for every $\gamma \in \Gamma_m$, \[\gamma(D_m)\cap (N\times \mathbb{R})\ne \emptyset .\]
		By assumption $(I_2)$, we have
		\begin{equation*}
			a= {\max}\; \pi_{\mathbb{R}} (S_m)\le 0 < \underset{u\in N}{\inf}\; F(u).
		\end{equation*}
		For every $\gamma \in \Gamma_m$, there exists $\mathscr{P} = \gamma(D_m) $ such that 
		\begin{equation*}
			{\max}\; \pi_{\mathbb{R}} (\mathscr{P})\ge \max\; \pi_{\mathbb{R}}\,\Big((N\times \mathbb{R})\cap \mathscr{P}\Big)\ge \underset{u\in N}{\inf}\; F(u).
		\end{equation*}
		Therefore,
		\begin{equation*}
			c= \underset{\gamma \in \Gamma}{\inf}\;{\max}\; \pi_{\mathbb{R}} (\mathscr{P}) \ge \underset{u\in N}{\inf}\; F(u).
		\end{equation*}
	Since \[c\ge a= {\max}\; \pi_{\mathbb{R}} (S_m),\] then $F$ satisfies the relation $(\ref{eq 4})$ of Theorem $\ref{theo 3.2}$ with \[\Gamma_0^m =\Big\{\operatorname{id}\Big\} \;\; \text{and}\;\; \gamma_0(S_m)= \operatorname{id}(S_m)=S_m.\] By applying Theorem $\ref{theo 3.3}$, we obtain the existence of a sequence $(u_n,c_n)\subset \mathbf{\mathit{graph}\,F}$ satisfying $(\ref{eq 9})$
		\end{proof}

	\section{Application of Theorem $\ref{theo 4.2}$}
	In this section, by applying Theorem $\ref{theo 4.2}$, we give an existence result for the following problem:
	\begin{equation}
		\label{eq 11}
		\begin{cases}
			-\Delta u(t) +\varphi(t)u(t)\in G(u(t)), \;\; &\text{a.e.}\;\; t\in \Omega,\\
			u(t)=0, \;\; &\text{for every}\;\; t\in \partial \Omega,
		\end{cases}
	\end{equation}
	where $\Omega$ is a bounded open subset of $\mathbb{R}^{N}$ with smooth boundary $\partial \Omega$, $N\ge3$, $\varphi \in L^{N/2}(\Omega)$ and $G :\mathbb{R} \to \mathbb{R}$ is a multivalued mapping.
	
	It is important to point out that this type of problem $(\ref{eq 11})$ arise naturally in control theory and allows one to study partial differential equations whose right-side is discontinuous.\\
	
	In $H_0^1(\Omega)$, we choose the norm
	\[\|u\|^2:= \int_{\Omega} |\nabla u(t)|^2dt.\]
	\begin{definition}
		\label{def 5.1}
		A weak solution of $(\ref{eq 11})$ is a function $u\in E:= H^1_0(\Omega)$ such that there exists a measurable function $h(t)\in G(u(t))$ a.e. $t\in \Omega$ and 
		\[\int_{\Omega} \Big(\nabla u(t)\cdot
		 \nabla \psi(t) + \varphi(t)u(t) \psi(t)\Big)dt -\int_{\Omega} h(t)\psi(t)\,dt =0, \;\; \text{for every}\;\; \psi \in E.\]
	\end{definition}
	
	Here is the main result of this section. It is inspired of \cite[Theorem 2.18]{Wi} and \cite[Theorem 4.1]{Fri}.
	\begin{theorem}
		\label{theo 5.1}
	Let $G: \mathbb{R} \to \mathbb{R}$ be an upper semi-continuous multivalued mapping with 
	nonempty, compact, convex values. Assume 
	\begin{enumerate}
		\item[$(G_1)$] the function $\varphi \in L^{N/2}(\Omega)$ if $N\ge 3$, $\varphi \in L^q(\Omega)$, $q>1$, if $N=2$ and $\varphi\in L^1(\Omega)$ if $N=1$, and
		 there exist $1<\mu<\frac{N+2}{N-2}$, and constants $a$ and $b$ such that
		\[|G(x)|=\max \Big\{|z|\;\Big|\;z\in G(x)\Big\} \le a+b|x|^\mu, \;\; \text{for every}\;\; x\in \mathbb{R};\]
		\item[$(G_2)$] there exist $\beta>2$ and $R>0$ such that for every $x$ such that $|x|\ge R$,
			\begin{equation*}
		0<\beta \int_{0}^{x}G(z)dz \le \min\big(xG(x)\big)=  \underset{z\in G(x)}{\min} \{xz\}=
			\begin{cases}
				x \underset{z\in G(x)}{\min} \{z\}\;\; \text{if}\;\; x>0,\\
				x \underset{z\in G(x)}{\max} \{z\}\;\; \text{if}\;\; x<0;
			\end{cases}
		\end{equation*}  
		\item[$(G_3)$] The following inequalities are satisfied :
		\begin{enumerate}
		\item[$(1)$] 	\[\underset{x\to 0}{\limsup}\; \frac{\max\;(xG(x))}{x^2}\le 0,\] where
		\begin{equation*}
			\max\big(xG(x)\big) =
			\begin{cases}
				x \underset{z\in G(x)}{\max} \{z\}\;\; \text{if}\;\; x>0,\\
				x \underset{z\in G(x)}{\min} \{z\}\;\; \text{if}\;\; x<0;
			\end{cases}
		\end{equation*}  
		\item[$(2)$] \[\lambda_{n} \frac{x^2}{2}\le \int_{0}^{x}\min \big(G(z)\big)dz,\] where $\lambda_1<\lambda_2\le \cdots\le \lambda_{n}\le 0<\lambda_{n+1}\le \cdots$ are the sequence of eigenvalues of
		\begin{equation}
			\label{eq 12}
			\begin{cases}
				-\Delta u(t) +\varphi(t)u(t)=\lambda u(t),\\
				u(t)\in H_0^1(\Omega),
			\end{cases}
		\end{equation}
		and where each eigenvalue is repeated according to its multiplicity.
		\end{enumerate}
	\end{enumerate}
	
	Then, the problem $(\ref{eq 11})$ has a nontrivial solution.
	\end{theorem}
	\begin{example}
		The set of functions $G$ that satisfies the conditions $(G_1)- (G_3)$ is nonempty. Choose $2<\alpha< \frac{2N}{N-2}$, $0<\alpha_1 \le \alpha_2$. Define
		\begin{equation*}
			G(x)=
			\begin{cases}
			[ \alpha_1 |x|^{\alpha-1},  \alpha_2 |x|^{\alpha-1}], \;\; &x\ge 0,\\
			[ -\alpha_2 |x|^{\alpha-1},  -\alpha_1 |x|^{\alpha-1}], \;\; &x< 0.
			\end{cases}
		\end{equation*}
		Then $G$ satisfies conditions $(G_1)- (G_3)$.
	\end{example}
	To prepare for the proof of Theorem $\ref{theo 5.1}$, we establish some lemmas.
	\begin{lemma}[Lemma 2.13, \cite{Wi}]
		\label{lem 5.1}
		If $N\ge3$ and $\varphi \in L^{N/2}(\Omega)$, the functional
		\[\mathscr{X}: H_0^1(\Omega)\to \mathbb{R} : u\mapsto \int_{\Omega} \varphi(t)u^2(t)\,dt \]
		is weakly continuous.
	\end{lemma}
	\begin{lemma}[Lemma 2.14, \cite{Wi}]
		\label{lem 5.2}
		If $|\Omega|<\infty$, $N\ge 3$ and $\varphi\in L^{N/2}(\Omega)$, then 
		\[\lambda_1:= \underset{\underset{\|u\|_{L^2(\Omega)}=1}{u\in H^1_0(\Omega)}}{\inf}\; \int_{\Omega} \Big(|\nabla u(t)|^2 +\varphi(t)u^2(t)\Big)dt >-\infty.\]
	\end{lemma}
	\begin{lemma}[Lemma 2.15, \cite{Wi}]
		\label{lem 5.3}
	Let $e_1$, $e_2$, $e_3$, $\cdots$ be the corresponding orthonormal eigenfunctions in $L^2(\Omega)$ of the eigenvalues $\lambda_1$, $\lambda_2$, $\cdots$ in $(\ref{eq 12})$.	Under the assumptions of Lemma $\ref{lem 5.1}$, if
		\[Y:=\spana \;(e_1,\cdots, e_{n}), \quad Z:= \Big\{  u\in H^1_0(\Omega)\;\Big|\; \int_{\Omega}u(t)v(t)\,dt=0, \; v\in Y\Big\},\] then
		\[\zeta:= \underset{\underset{\|\nabla u\|_{L^2(\Omega)}=1}{u\in Z}}{\inf}\;\int_{\Omega} \Big(|\nabla u(t)|^2 +\varphi(t)u^2(t)\Big)dt>0. \]
	\end{lemma}
	\begin{remark}
		\label{rm 3}
		As a consequence of Lemmas $\ref{lem 5.1}$ and $\ref{lem 5.2}$, if $u=u_Y+u_Z$, $u_Y\in Y$ and $u_Z\in Z$, then
		\[A(u_Y) := \int_{\Omega} \Big(|\nabla u_Y(x)|^2 +\varphi(x)u_{Y}^2(x)\Big)dx= \sum_{i=1}^{n}\lambda_i\|u_Y\|^2_{L^2(\Omega)},\]
		\[A(u) := \int_{\Omega} \Big(|\nabla u(x)|^2 +\varphi(x)u^2(x)\Big)dx\ge \lambda_1\|u_Y\|^2_{L^2(\Omega)} +\zeta \|u_Z\|^2.\]
	\end{remark}
We introduce some notation that will be used throughout this section:
\begin{equation*}
	\mathcal{S}(G) = \Big\{g: \mathbb{R} \to \mathbb{R} \;|\; g\; \text{is measurable, and}\; g(x)\in G(x)\; \text{a.e.}\; x\in \mathbb{R}\Big\},
\end{equation*}
The set $\mathcal{S}(G)$ is nonempty by the Ryll-Nardzewski measurable selection theorem.
\begin{equation*}
	\mathcal{S}(u) =\Big\{h: \Omega\to \mathbb{R} \;|\; h\;\: \text{is measurable, and}\;\: h(t)\in G(u(t))\; \text{a.e.}\; t\in \Omega\Big\}.
\end{equation*}
It is well known that the set $\mathcal{S}(u)$ is nonempty.

We denote
\begin{equation*}
	\underline{g}(x) :=\min \{G(x)\},\quad \overline{g}(x) :=\max \{G(x)\}.
\end{equation*}

For each $g\in \mathcal{S}(G)$, we define the functional
\begin{eqnarray*}
	I_g:H_0^1(\Omega) &\to& \mathbb{R}\\
	u&\mapsto& \frac{1}{2}\int_{\Omega} \Big(|\nabla u(t)|^2 + \varphi(t) u^2(t) \Big)dt -\int_{\Omega}\Big( \int_{0}^{u(t)} g(x)dx \Big)dt.
\end{eqnarray*}
	By assumption $(G_1)$, the functional $I_g$ is well defined and continuous.
	
	Let us define the multivalued functional
	\begin{eqnarray}
		\label{eq 13}
		F: H_0^1(\Omega) &\to& \mathbb{R}\nonumber\\
		u &\mapsto& \Big\{c\in \mathbb{R} \;|\; I_g(u)\le c\; \;\text{for some}\; g\in \mathcal{S}(G)\Big\}.
	\end{eqnarray}
	
	In the sequel, we denote $E:=H_0^1(\Omega)$.
	\begin{lemma}
		\label{lem 5.4}
Under assumption $(G_1)$, the functional $F$ defined in $(\ref{eq 13})$ has 
	closed graph. 
	\end{lemma}
	\begin{proof}
		Let $\Big\{(u_n, c_n)\Big\} \subset \mathbf{\mathit{graph}\,F}:=\Big\{(v,b)\in E\times \mathbb{R}\;|\; b\in F(v)\Big\}$, be a sequence such that $c_n \to c\in \mathbb{R}$, and $u_n \to u \in E$. We want to show that $(u,c)\in \mathbf{\mathit{graph}\,F}$.\\
		There exists $\{g_n\}\subset\mathcal{S}(G)$ such that $I_{g_n}(u_n)\le c_n$. Since $G$  is upper semi-continuous with compact values, we have that $G([-k,k])$ is compact (see \cite[Theorem 1.2.15]{BGMO}) for every $k>0$. Since $g_n(x)\in G(x)$ a.e. $x\in \mathbb{R}$, then for a.e. every $x\in [-k,k]$, $|g_n(x)|<\infty$. Therefore, for $r\in (1,+\infty)$ fixed, we have
		\begin{equation*}
			\|g_n\|_{L^r([-k,k])}^{r}= \int_{-k}^{k}|g_n(x)|^r dx < \infty.
		\end{equation*}
		It follows that $\{g_n\}$ is bounded in $L^r([-k,k])$, for every $k>0$. Since  $L^r([-k,k])$ is a reflexive space, there exist $g \in L^r_{loc}(\mathbb{R})$ and a subsequence of $\{g_n\}$, which we still denote by $\{g_n\}$, such that
		\[ g_n \rightharpoonup g \;\; \text{in} \;\; L^r([-k,k]), \;\; \text{for all} \;\; k>0.\]
		 Since every convex subset is weakly closed if and only if it is strongly closed, we also have
		\begin{equation*}
			g(x)\in \overline{\text{conv}}\;\{g_n(x), g_{n+1}(x),\cdots\} \;\; a.e.\;\; x\in \mathbb{R},
		\end{equation*}
		where \[\overline{\text{conv}}\;\{g_n(x), g_{n+1}(x),\cdots\},\] denotes the closure of the convex hull of $\{g_n(x), g_{n+1}(x),\cdots\}$; that is, the intersection of all convex sets containing $g_n(x), g_{n+1}(x), \cdots$. This with the fact that $\{g_n\}\subset \mathcal{S}(G)$ and $G$ is semi-continuous with compact, convex values imply that $g(x)\in G(x)$ a.e. $x\in \mathbb{R}$.
		
		By the Sobolev embedding theorem, \[u_n\to u\;\; \text{in}\;\; L^p(\Omega),\;\; \text{for every}\;\; 2\le p <\infty.\]  Up to a subsequence, we can assume that
		\begin{equation*}
			u_n(t) \to u(t), \;\; \text{a.e.}, \;\; \text{on}\;\; \Omega, \;\; \text{whenever}\;\; n \to \infty.
		\end{equation*}
		
Let $f \in \mathcal{S}(G)$, by assumption $(G_1)$, for every compact $K\subset \mathbb{R}$, we have 
\[\int_{K}|f(x)|dx \le \int_{K} \Big(a +b|x|^\mu\Big) dx< \infty,\] that is, $f \in L^1_{loc}(\mathbb{R})$. Then, for fixed $y_0\in \mathbb{R}$, we know that the function 
\[\tau \mapsto \int_{y_0}^{\tau}f(x)dx\;\; \text{is continuous, for every}\;\; \tau\in \mathbb{R}.\]
	Hence, by continuity and the fact that $g_n \rightharpoonup g$ in $L^r([-k,k])$ for all $k>0$, we have
		\begin{equation*}
			\int_{0}^{u_n(t)} g_n(x)dx \to \int_{0}^{u(t)} g_n(x)dx \to \int_{0}^{u(t)}g(x)dx, \;\; \text{a.e.}, \;\; \text{on}\;\; \Omega.
		\end{equation*}
	The Lebesgue convergence dominated theorem implies that 
		\begin{equation*}
			\int_{\Omega}\Big(\int_{0}^{u_n(t)} g_n(x)dx\Big)dt \to \int_{\Omega}\Big(\int_{0}^{u(t)}g(x)dx\Big)dt.
		\end{equation*}
		Thus,
		\begin{equation*}
			I_g(u)= \lim_{n\to \infty} I_{g_n}(u_n)\le \lim\limits_{n\to \infty} c_n=c,
		\end{equation*}
that is, $(u,c)\in \mathbf{\mathit{graph}\,F}$.
	\end{proof}
	
		For $u\in E$, define 
	\begin{multline*}
		\mathscr{A}(u):= \Big\{\alpha\in E^*\;|\; \exists h\in \mathcal{S}
		(u),\; \text{such that}\;\\
		\langle \alpha, w\rangle = \int_{\Omega} \Big(\nabla u(t)\cdot \nabla w(t) +\varphi(t) u(t)w(t)\Big)dt - \int_{\Omega} \big(h(t)w(t)\big)dt,\;\; \text{for every}\;\; w\in E \Big\},
	\end{multline*}
	where $\Big(E^*, \|\cdot\|_*\Big)$, is the topological dual of $E$ and $\langle \cdot, \cdot \rangle$, denotes the duality pairing. 
	
	Observe that, if $0\in \mathscr{A}(u)$, then $u$ is a weak solution of the problem $(\ref{eq 11})$ (see Definition $\ref{def 5.1}$).\\
	We also have $\mathscr{A}(u)\ne \emptyset$. Indeed, let $h\in \mathcal{S}(u)$. Then, the map
	\begin{equation*}
		L(w):= \int_{\Omega} \Big(\nabla u(t) \cdot\nabla w(t) +\varphi(t)u(t)w(t)\Big)dt - \int_{\Omega} \big(h(t)w(t)\big)dt, \; w \in E,
	\end{equation*}
	is well defined, linear and continuous. So, for $\alpha \in E^*$, we have
	\begin{equation*}
		L(w)=\langle \alpha, w\rangle.
	\end{equation*}
	
	\begin{lemma}
		\label{lem 5.5 }
		Let $u\in E$. Then, there exist $\alpha\in \mathscr{A}(u)$ and $w\in E$ with $\|w\|\le 1$ such that
		\begin{equation*}
			\|\alpha\|_* =\langle \alpha, w\rangle \le \langle \overline{\alpha}, w\rangle, \;\; \text{for every}\;\;  \overline{\alpha}\in \mathscr{A}(u).
		\end{equation*}
	\end{lemma}
	\begin{proof}
		Let $u\in E$, $\alpha\in \mathscr{A}(u)$. Then, there exists $h\in \mathcal{S}(u)$ such that, for every  $w\in E$, we have
		\begin{equation*}
			\langle \alpha, w\rangle = \int_{\Omega} \Big(\nabla u(t) \cdot\nabla w(t) +\varphi(t)u(t)w(t)\Big)dt - \int_{\Omega} \big(h(t)w(t)\big)dt.
		\end{equation*}
		
		 By definition, 
		\begin{equation*}
			\|\alpha\|_* := \underset{\underset{\|w\|\le 1}{w\in E}}{\sup}\; \big| \langle \alpha, w\rangle\big|.
		\end{equation*}
		
		By assumption $(G_1)$, it follows that the set $\mathscr{A}(u)$ is bounded. Moreover, since $G$ has convex values, then $\mathscr{A}(u)$ is convex. Indeed, let $\alpha_1$, $\alpha_2 \in \mathscr{A}(u)$. Then, there exist $h_1$, $h_2 \in \mathcal{S}(u)$ such that for every $\lambda \in [0,1]$, we have
		\begin{equation*}
			\langle \lambda \alpha_1 + (1-\lambda)\alpha_2, w \rangle = \lambda \; \langle \alpha_1, w\rangle + (1-\lambda)\;\langle \alpha_2, w\rangle,
		\end{equation*}
		with 
		\begin{equation*}
			\lambda \; \langle \alpha_1, w\rangle = \lambda \; \Big( \int_{\Omega} \big(\nabla u(t) \cdot \nabla w(t) +\varphi(t)u(t)w(t)\big)dt - \int_{\Omega} \big(h_1(t)w(t)\big)dt\Big),
		\end{equation*}
		and
		\begin{equation*}
			(1-\lambda)\;\langle \alpha_2,w\rangle = (1-\lambda)\; \int_{\Omega} \big(\nabla u(t)\cdot \nabla w(t) +\varphi(t)u(t)w(t)\big)dt -\int_{\Omega}       (1-\lambda) \big (h_2(t)w(t)\big)dt.
		\end{equation*}
		We obtain
		\begin{equation*}
			\langle \lambda \alpha_1 + (1-\lambda)\alpha_2, w \rangle = \Big( \int_{\Omega} \big(\nabla u(t) \cdot \nabla w(t) +\varphi(t)u(t)w(t)\big)dt - \int_{\Omega} \Big(\lambda h_1(t) + (1-\lambda)h_2(t)\Big)w(t)  \,dt.
		\end{equation*}	 
	But $G$ has convex values, so
		\begin{equation*}
			\lambda h_1(t) + (1-\lambda)h_2(t) \in \mathcal{S}(u).
		\end{equation*}
		This implies that  $\lambda \alpha_1 + (1-\lambda)\alpha_2 \in \mathscr{A}(u)$.\\
		
		Let us consider the map 
		\begin{equation*}
			\mathcal{Q} : \mathscr{A}(u) \times B_1 \to \mathscr{A}(u)\times B_1
		\end{equation*}
		defined by
		\begin{equation*}
			\mathcal{Q} \Big(\tilde{\alpha},\Tilde{w}\Big) =\Big\{ \big(\alpha,w\big)\in \mathscr{A}(u) \times B_1\;\Big|\; \langle  \Tilde{\alpha}, w\rangle - \langle \alpha, \Tilde{w}\rangle \ge 0 \Big\},
		\end{equation*}
		where $B_1$ is the unit closed ball centered at $O$ in $E$. 
		
		\textbf{For $(\tilde{\alpha},\Tilde{w}) \in \mathscr{A}(u) \times B_1$, the set $\mathcal{Q}\Big(\tilde{\alpha},\Tilde{w}\Big)$ is closed and convex}:\\
		$\bullet$ \; Let $(\alpha_m,w_m) \in \mathcal{Q} \Big( \tilde{\alpha},\Tilde{w}\Big)$ be a sequence that converges to $(\alpha,w) \in \mathscr{A}(u) \times B_1 $.
		By continuity of $\tilde{\alpha}$, we have
		\begin{equation*}
			\langle  \Tilde{\alpha},w\rangle - \langle \alpha,\Tilde{w}\rangle = \lim\limits_{m\to \infty} \; \langle  \Tilde{\alpha}, w_m \rangle - \lim\limits_{m\to \infty}\; \langle \alpha_m,\Tilde{w}\rangle \ge 0,
		\end{equation*}
	which implies that $ (\alpha, w) \in \mathcal{Q} \Big( \tilde{\alpha},\Tilde{w}\Big)$. So $\mathcal{Q}\Big( \tilde{\alpha},\Tilde{w}\Big)$ is closed.\\
		
		$\bullet$ Let $(\alpha_1,w_1), (\alpha_2,w_2) \in \mathcal{Q}\Big( \tilde{\alpha},\Tilde{w}\Big)$, $\lambda\in[0,1]$.
		We have
		\begin{multline*}
			\langle  \Tilde{\alpha},\lambda w_1 +(1-\lambda)w_2\rangle - \langle \lambda \alpha_1 +(1-\lambda)\alpha_2, \Tilde{w} \rangle =\\
			\lambda \Big( \langle \Tilde{\alpha},w_1\rangle - \langle \alpha_1, \Tilde{w} \rangle\Big) + (1-\lambda) \Big( \langle \Tilde{\alpha},w_2\rangle - \langle \alpha_2,\Tilde{w} \rangle\Big).
		\end{multline*}
		But
		\begin{equation*}
			\langle \Tilde{\alpha},w_1\rangle - \langle \alpha_1, \Tilde{w}\ge 0 \;\; \text{and}\;\; 	\langle \Tilde{\alpha},w_2\rangle - \langle \alpha_2,\Tilde{w} \rangle \ge 0.
		\end{equation*}
		We conclude that
		\begin{equation*}
			\lambda \Big( \langle \Tilde{\alpha},w_1\rangle - \langle \alpha_1, \Tilde{w} \rangle\Big) + (1-\lambda) \Big( \langle \Tilde{\alpha},w_2\rangle - \langle \alpha_2,\Tilde{w} \rangle\Big) \in \mathcal{Q}\Big( \tilde{\alpha},\Tilde{w}\Big),
		\end{equation*}
		that is,
		\begin{equation*}
			\lambda (\alpha_1,w_1)+ (1-\lambda)(\alpha_2,w_2) \in \mathcal{Q}\Big( \tilde{\alpha},\Tilde{w}\Big),
		\end{equation*}
		that is, $\mathcal{Q}\Big( \tilde{\alpha},\Tilde{w}\Big)$ is convex.
		
		\textbf{The map $\mathcal{Q}$ is KKM}.
		Indeed, if there exist 
		\begin{equation*}
			(\tilde{\alpha}_1, \Tilde{w}_1), \cdots, (\tilde{\alpha}_n,\Tilde{w}_n)
		\end{equation*}
		such that 
		\begin{equation*}
			\text{conv}\; \Big\{ (\tilde{\alpha}_1,\Tilde{w}_1) \cdots, (\tilde{\alpha}_n,\Tilde{w}_n) \Big\} \not\subset \bigcup_{i=1}^{n} \mathcal{Q}\Big(\tilde{\alpha}_i,\Tilde{w}_i\Big),
		\end{equation*}
	then there exist $\lambda_1, \cdots, \lambda_n \in [0,1]$ such that $\sum_{i=1}^{n}\lambda_i =1$, and
		\begin{equation*}
			\sum_{i=1}^{n}\lambda_i (\tilde{\alpha}_i,\Tilde{w}_i) \notin \mathcal{Q}\Big(\tilde{\alpha}_j,\Tilde{w}_j\Big),\;\; \text{for}\;\; j=1, \cdots, n.
		\end{equation*}
	That is,
		\begin{equation*}
			\langle \tilde{\alpha}_j, \sum_{i=1}^{n}\lambda_i  \Tilde{w}_i \rangle - \langle  \sum_{i=1}^{n}\lambda_i  \Tilde{\alpha}_i,\Tilde{w}_j \rangle <0, \;\; \text{for }\;\; j=1,\cdots, n.
		\end{equation*}
		
		This implies that 
		\begin{equation*}
			0= \sum_{j=1}^{n}\lambda_j \Bigg[ \langle \tilde{\alpha}_j, \sum_{i=1}^{n}\lambda_i  \Tilde{w}_i\rangle - \langle  \sum_{i=1}^{n}\lambda_i  \Tilde{\alpha}_i,  \Tilde{w}_j \rangle\Bigg] <0,
		\end{equation*}
		which is a contradiction.
		 
	 By the Elementary KKM Principle \cite[Theorem 5.2]{GL} there exists $(\alpha, w) \in \mathcal{Q}\Big(\tilde{\alpha},\Tilde{w}\Big)$ for every $(\tilde{\alpha},\Tilde{w}) \in \mathscr{A}(u) \times B_1.$
	That is, for every  $(\tilde{\alpha},\Tilde{w}) \in \mathscr{A}(u) \times B_1$, we have
		\begin{equation*}
			\langle \Tilde{\alpha},w\rangle \ge \langle \alpha,\Tilde{w}\rangle.
		\end{equation*}
	Hence, there exist 
		$\alpha\in \mathscr{A}(u)$ and $w\in E$ with $\|w\|\le 1$ such that
		\begin{equation*}
			\|\alpha\|_* =\langle \alpha,w\rangle \le \langle \overline{\alpha},w\rangle, \;\; \text{for every}\;\;  \overline{\alpha}\in \mathscr{A}(u).
		\end{equation*}
	\end{proof}
	The following lemma ensures that the critical points of $F$ correspond to the weak solutions of the problem $(\ref{eq 11})$.
	\begin{lemma}
		\label{lem 5.6}
		Let $u\in E$ and $c\in F(u)$. Then, there exist $\alpha\in \mathscr{A}(u)$ such that
		\begin{equation*}
			\|\alpha\|_* \le |dF|(u,c).
		\end{equation*}
	\end{lemma}
	\begin{proof}
		If $0\in \mathscr{A}(u)$, then $u$ si a weak solution of problem $(\ref{eq 11})$. If $|dF|(u,c) =\infty$, then $\|\alpha\|_*< \infty.$\\
		Suppose $|dF|(u,c) <\infty$ and $0\notin \mathscr{A}(u)$. By Lemma $\ref{lem 5.5 }$, there exist $\alpha\in \mathscr{A}(u)$ and $\hat{w}\in E$ with $\|\hat{w}\|\le 1$ such that
		\begin{equation*}
			\|\alpha\|_* =\langle \alpha,\hat{w}\rangle \le \langle \overline{\alpha},\hat{w}\rangle, \;\; \text{for every}\;\;  \overline{\alpha}\in \mathscr{A}(u).
		\end{equation*}
	If $\|\alpha\|_* =0$, the proof is complete. If $\|\alpha\|_* >0$,	let
			$\varepsilon < \frac{\|\alpha\|_*}{2}$.
		There exists $w\in \mathcal{C}_0^{\infty}(\Omega)$ with $\|w\| \le 1$ and $\|\hat{w}-w\|$ sufficiently small such that
		\begin{equation}
			\label{eq 14}
			\|\alpha\|_* -\varepsilon < \langle \overline{\alpha},w\rangle =\int_{\Omega} \Big(\nabla u(t) \cdot\nabla w(t) +\varphi(t)u(t)w(t) \Big)dt - \int_{\Omega} \overline{h}(t)w(t) \,dt,\;\; \text{for every}\;\;  \overline{\alpha}\in \mathscr{A}(u).
		\end{equation}
		Moreover, there is $\delta_1>0$ such that for every $(u_0,b)\in B_{\delta_1}(u,c)\cap \mathbf{\mathit{graph}\,F}$, for every $s\in [0, \delta_1]$, we have
		\begin{multline}
			\label{eq 15}
			s \Bigg(\int_{\Omega} \big(\nabla u_0(t)\cdot \nabla w(t) +\varphi(t)u_0(t)w(t)\big) dt \Bigg) - \int_{\Omega}\Big( \int_{u_0(t) -sw(t)}^{u_0(t)}g(x)dx \Big) dt \ge s \big(\|\alpha\|_* -\varepsilon \big),\\ \text{for every}\;\; g\in \mathcal{S}(G)\;\; \text{such that}\;\; I_g(u_0)\le b.
		\end{multline}
		If not, there would exist $(u_n,b_n)\to (u,b)$ in $B_{\delta_1}(u,c)\cap \mathbf{\mathit{graph}\,F}$, $s_n\to 0$, and $g_n \in \mathcal{S}(G)$ such that $I_{g_n}(u_n)\le b_n$ and
		\begin{equation}
			\label{eq 16}
			\|\alpha\|_* -\varepsilon > \Bigg(\int_{\Omega} \big(\nabla u_n(t)\cdot \nabla w(t) +\varphi(t)u_n(t)w(t)\big) dt  \Bigg) 
			-\dfrac{1}{s_n} \int_{\Omega}\Big( \int_{u_n(t) -s_nw(t)}^{u_n}g_n(x)dx \Big) dt.
		\end{equation}
		On $\big\{ t\;|\; w(t)\ge 0 \big\}$,
		since $g_n(x)\le \overline{g}(x)$, then we have 
		\begin{equation*}
			\int_{u_n(t) -s_nw(t)}^{u_n(t)}g_n(x)dx \le \int_{u_n(t) -s_nw(t)}^{u_n(t)}\overline{g}(x)dx.
		\end{equation*}
		On $\big\{ t\;|\; w(t)< 0\big\}$,
		since $g_n(x)\ge \underline{g}(x)$, we have
		\begin{equation*}
			\int_{u_n(t) -s_nw(t)}^{u_n(t)}g_n(x)dx =- \int^{u_n(t) -s_nw(t)}_{u_n(t)} g_n(x)dx  \le - \int^{u_n(t) -s_nw(t)}_{u_n(t)} \underline{g}(x)dx = \int_{u_n(t) -s_nw(t)}^{u_n(t)}\underline{g}(x)dx.
		\end{equation*}
		By $(\ref{eq 16})$, we obtain
		\begin{multline*}
			\|\alpha\|_* -\varepsilon > \Bigg(\int_{\Omega} \big(\nabla u_n(t) \cdot\nabla w(t) +\varphi(t)u_n(t)w(t)\big) dt \Bigg) \\
			-\dfrac{1}{s_n} \int_{\big\{ t\;|\; w(t)\ge 0\big\}}\Big( \int_{u_n(t) -s_nw(t)}^{u_n(t)}\overline{g}(x)dx \Big) dt \\
			-	\dfrac{1}{s_n} \int_{\big\{ t\;|\; w(t)< 0 \big\}}\Big( \int_{u_n(t) -s_nw(t)}^{u_n(t)}\underline{g}(x)dx \Big) dt.	
		\end{multline*}
	By passing to the limit as $n \to \infty$, we have
	
		\begin{multline*}
			\|\alpha\|_* -\varepsilon > \int_{\Omega} \big(\nabla u(t) \cdot\nabla w(t) +\varphi(t)u(t)w(t)\big) dt
			- \int_{\big\{ t\;|\; w(t)\ge 0 \big\}} \Big(\overline{g}(u)w(t) \Big)dt
			- \int_{\big\{ t\;|\; w(t)< 0  \big\}} \Big(\underline{g}(u)w(t)\Big)dt\\
			= \int_{\Omega} \Big(\nabla u(t) \cdot\nabla w (t)+\varphi(t)u(t)w(t)\Big) dt -\int_{\Omega} \Big(\hat{g}(t)w(t)\Big)dt\\ > 	\|\alpha\|_* -\varepsilon \;\; \text{(by $(\ref{eq 14})$)},
		\end{multline*}
		where 
		\begin{equation*}
			\hat{g}(t) =
			\begin{cases}
				\overline{g}(u(t)),\;\; \text{if}\; \;t\in\{t\;|\; w(t)\ge 0\},\\
				\underline{g}(u(t)), \;\; \text{if} \;t\in\{t\;|\; w(t)<0 \},
			\end{cases}
		\end{equation*}
		which is a contradiction.
	
		Let
		\begin{equation*}
			\delta:= \min \Bigg\{ \delta_1, \dfrac{2\varepsilon}{|\lambda|\,\|w\|_{L^2(\Omega)}^2}\Bigg\},
		\end{equation*}
		where $\lambda$ is the constant appearing in $(\ref{eq 12})$.
		
		Define
		\begin{eqnarray*}
			\mathcal{H} : B_{\delta}(u,c)\cap \mathbf{\mathit{graph}\,F} \times[0,\delta] &\to& \mathbf{\mathit{graph}\,F} \\
			\big((x,b), s\big) &\to& \Big(x-sw, \;b-s\big(	\|\alpha\|_* -2\varepsilon\big)\Big).
		\end{eqnarray*}
		The function $\mathcal{H}$ is well defined. Indeed, let $g\in \mathcal{S}(G)$ such that $I_g(x)\le b$. We have
		\begin{equation*}
			I_g(x-sw) = \dfrac{1}{2}\int_{\Omega} \Bigg(\Big|\nabla \Big(x(t)-sw(t)\Big)\Big|^2 +\varphi(t)\Big(x(t)-sw(t)\Big)^2\Bigg)dt -\int_{\Omega}\Big( \int_{0}^{x(t)-sw(t)} g(\tau)d\tau \Big)dt.
		\end{equation*}
		This implies that
		\begin{multline*}
			I_g(x-sw) = \Bigg( \dfrac{1}{2}\int_{\Omega} \Big(|\nabla x(t)|^2+\varphi(t)x^2(t) \Big)dt -\int_{\Omega}\Big( \int_{0}^{x(t)} g(\tau)d\tau \Big)dt\Bigg)\\ +\int_{\Omega} \dfrac{s^2}{2}\Big(|\nabla w(t)|^2+\varphi(t)w^2(t) \Big)dt
			- s \Bigg( \int_{\Omega} \Big(\nabla x (t)\cdot \nabla w(t) +\varphi(t)x(t)w(t) \Big)dt \Bigg) \\ + \int_{\Omega}\Big( \int_{x(t) -sw(t)}^{x(t)}g(\tau)d\tau  \Big) dt.
		\end{multline*}
		From $(\ref{eq 15})$, we get
		\begin{equation*}
			I_g(x-sw) \le I_g(x) - s\big(\|\alpha\|_* -\varepsilon \big) + \Big|\int_{\Omega} \dfrac{s^2}{2}\Big(|\nabla w(t)|^2+\varphi(t)w^2(t)\Big)dt\Big|.
		\end{equation*}
		Since $s\le \delta \le \frac{2\varepsilon}{|\lambda|\,\|w\|_{L^2(\Omega)}^2}$,
		then, $s |\lambda|\,\|w\|^2_{L^2(\Omega)}\le 2\varepsilon$.
		That is, $s^2 |\lambda|\,\|w\|_{L^2(\Omega)}^2\le 2s\varepsilon$. Since by the equation $(\ref{eq 12})$,
		\[\Big|\int_{\Omega} \dfrac{s^2}{2}\Big(|\nabla w(t)|^2+\varphi(t)w^2(t)\Big)dt\Big| = \dfrac{s^2}{2}|\lambda|\,\|w\|^2_{L^2(\Omega)}, \] we finally have
		\begin{equation*}
			I_g(x-sw) \le b - s\big(\|\alpha\|_* -\varepsilon \big) +s\varepsilon = b -s\big(\|\alpha\|_* -2\varepsilon \big).
		\end{equation*}
		Thus, \[\mathcal{H} \big((x,b), s\big) \in \mathbf{\mathit{graph}\,F}.\]
	By the Definition of weak slope of $F$ (Definition $\ref{def 2.4}$), we have that
		\begin{equation*}
			\|\alpha\|_* -2\varepsilon \le |dF|(u,c),
		\end{equation*}
		for every arbitrary $\varepsilon>0$. Therefore,
		\begin{equation*}
			\|\alpha\|_* \le |dF|(u,c).
		\end{equation*}
	\end{proof} 
	\begin{lemma}
		\label{lem 5.7}
		Assume that $(G_1)$ and $(G_2)$ are satisfied. Then, every Palais-Smale sequence at level $c\in \mathbb{R}$ is bounded.
	\end{lemma}
	\begin{proof}
		Let $(u_n)\subset E=Y\oplus Z$ (where $Y$ and $Z$ are as in Lemma $\ref{lem 5.3}$) and $c_n\in F(u_n)$ such that $c_n \to c $ and  $|dF|(u_n,c_n)\to 0$,  as $n\to \infty$. By Lemma $\ref{lem 5.6}$, there exists $\alpha_n\in \mathscr{A}(u_n)$ with
			$\|\alpha_n\|_* \le |dF|(u_n,c_n)$, where
		\begin{equation}
			\label{eq 17}
			\langle \alpha_n,w\rangle =\int_{\Omega} \Big(\nabla u_n(t)\cdot \nabla w(t) +\varphi(t)u_n(t)w(t)\Big)dt - \int_{\Omega} \big(h_n(t)w(t)\big)dt,
		\end{equation}
	for every $w\in E$.
	
	Let $\tau\in \Big(\frac{1}{\beta}, \frac{1}{2}\Big)$. For sufficiently large $n\in \mathbb{N}$, we have \[\|\alpha_n\|_* \le |dF|(u_n,c_n)\le1\;\; \text{and}\;\; 1+c_n \ge I_{g_n}(u_n).\] Hence, for such $n$, we have 
	\begin{equation*}
		c_n+1 + \|u_n\| \ge I_{g_n}(u_n)- \tau \langle \alpha_n, u_n\rangle.
	\end{equation*}
	We have
		\begin{align*}
			c_n+1 + \|u_n\| &\ge I_{g_n}(u_n) - \tau \langle \alpha_n, u_n\rangle \\ &= \dfrac{1}{2}\int_{\Omega} \Big(|\nabla u_n(t)|^2 +\varphi(t)u_n^2(t)\Big)dt -\int_{\Omega}\Big( \int_{0}^{u_n(t)} g_n(x)dx\Big)dt- \tau \langle \alpha_n, u_n\rangle \\
			&=\Big(\dfrac{1}{2}-\tau\Big)\int_{\Omega} \Big(|\nabla u_n(t)|^2 +\varphi(t)u_n^2(t) \Big)dt +\int_{\Omega}\Big(\tau h_n(t)u_n(t) - \int_{0}^{u_n(t)} g_n(x)dx\Big)dt.
		\end{align*}
		On the one hand, by assumption $(G_2)$, 
		\[\beta \tau \int_{0}^{u_n(t)}g_n(x)dx \le \tau u_n(t) h_n(t).\]
		Therefore,
		\begin{align}
			\label{eq 18}
			\tau h_n(t)u_n(t) - \int_{0}^{u_n(t)} g_n(x)dx &\ge - \int_{0}^{u_n(t)}g_n(x)dx +\beta \tau \int_{0}^{u_n(t)}g_n(x)dx\nonumber\\
			&= (\beta \tau-1)\int_{0}^{u_n(t)}g_n(x)dx.
		\end{align}
		On the other hand, by assumption $(G_2)$, there exist constants $k$, $l>0$ such that for every $t\in \mathbb{R}$ and every $g\in \mathcal{S}(G)$ (see \cite{Fri}), we have 
		\begin{equation}
			\label{eq 19}
			\int_{0}^{t}g(x)dx \ge k|t|^\beta -l.
		\end{equation}
		It follows that 
		\begin{equation}
			\label{eq 20}
		\int_{\Omega} \Big(\int_{0}^{u_n(t)}g_n(x)dx\Big)dt \ge k\|u_n\|^\beta_{L^\beta(\Omega)}-l_1,
		\end{equation}
		for some positive constant $l_1$.\\
		By Remark $\ref{rm 3}$, $u_n =y_n +z_n$, $y_n\in Y$, $z_n \in Z$ and
		\begin{equation}
			\label{eq 21}
			\int_{\Omega} \Big(|\nabla u_n(t)|^2 +\varphi(t)u_n^2(t) \Big)dt \ge \zeta \|z_n\|^2 +\lambda_1 \|y_n\|^2_{L^2(\Omega)}.
		\end{equation}
		Combining equations $(\ref{eq 18})$, $(\ref{eq 20})$ and $(\ref{eq 21})$, we obtain
		\begin{equation*}
			c_n+1 + \|u_n\|\ge \Big(\dfrac{1}{2}-\tau\Big) \Big(\zeta \|z_n\|^2 +\lambda_1 \|y_n\|^2_{L^2(\Omega)}\Big)+k(\beta\tau -1)\|u_n\|^\beta_{L^\beta(\Omega)} -l_2,
		\end{equation*}
		for some positive constant $l_2$.\\
		Since $\dim Y$ is finite, $\beta>2$ and $k(\beta\tau -1)>0$, it follows that $\|u_n\|<\infty$. 
	\end{proof}
 Let us show that the multivalued functional $F$ verifies conditions $(I_1)$ and $(I_2)$ in Theorem $\ref{theo 4.2}$.
	\begin{lemma}
		\label{lem 5.8}
		Under assumptions $(G_1)-(G_3)$, the multivalued functional $F$ satisfies conditions $(I_1)-(I_2)$ in Theorem $\ref{theo 4.2}$.
	\end{lemma}
	\begin{proof}
	1. \textbf{We first show that the multivalued functional $F$ satisfies condition $(I_1)$}.\\
	
	Let 
 \[Y= \spana\;\{e_1, \cdots, e_{n}\} \;\; \text{and}\;\; Z= Y^\perp.\] 
 From assumption $(G_4)$, for every $\varepsilon>0$, there exits $\delta>0$ such that for every $g(x)\in G(x)$, we have 
 \[g(x)\le 2\varepsilon x\;\; \text{if}\;\; x\in[0,\delta]\;\; \text{and}\;\; g(x)\ge 2\varepsilon x\;\; \text{if}\;\; x\in[-\delta, 0].\]
 The assumption $(G_1)$ implies the existence of a positive constant $a_1$ such that 
 \[|G(x)|\le a_1|x|^\mu \;\; \text{if}\;\: |x|\ge \delta.\]
 Thus, for every $t\in \mathbb{R}$ and every $g\in \mathcal{S}(G)$, we have 
 \[\int_{0}^{t}g(x)dx\le \varepsilon t^2+ a_2 |t|^{\mu +1}.\]
From Remark $\ref{rm 3}$, if $u =y +z$, $y\in Y$, $z \in Z$ then
	\begin{equation*}
		\int_{\Omega} \Big(|\nabla u(t)|^2 +\varphi(t)u^2(t) \Big)dt \ge \zeta \|z\|^2 +\lambda_1 \|y\|^2_{L^2(\Omega)}.
	\end{equation*}
	We deduce that, for every $u\in Z$, 
	\begin{align*}
		I_g(u)&\ge \frac{\zeta}{2}\|u\|^2 -\varepsilon \int_{\Omega} |u(t)|^2dt -a_2 \int_{\Omega}|u(t)|^{\mu +1}dt\\
		&=\frac{\zeta}{2}\|u\|^2 -\varepsilon \|u\|^2_{L^2(\Omega)}-a_2 \|u\|^{\mu+1}_{L^{\mu+1}(\Omega)}.
	\end{align*}
	By Sobolev embedding theorem, 
	\begin{align*}
		I_g(u)&\ge \frac{\zeta}{2}\|u\|^2 - \varepsilon\overline{c_1}\|u\|^2 -\overline{c_2}\|u\|^{\mu+1}\\
		&=\Big(\frac{\zeta}{2}-\varepsilon\overline{c_1}\Big)\|u\|^2 -\overline{c_2}\|u\|^{\mu+1},
	\end{align*}
	for some constants $\overline{c_1}>0$, $\overline{c_2}>0$.\\
	 We may choose $\varepsilon$ so that $\frac{\zeta}{2}-\varepsilon\overline{c_1}>0$.
	For every $u\in Z$, we can find $\nu>0$ and $r>0$ sufficiently small such that
	\begin{equation}
		\label{eq 22}
		I_g(u)\ge \nu>0, \;\; \text{whenever}\;\; \|u\|=r.
	\end{equation}
	That is to say, we have demonstrated that $\underset{u\in N}{\inf}\; F(u) >0$, where $N$ is defined in Theorem $\ref{theo 4.2}$.\\
	
2. \textbf{Let us show that the multivalued functional $F$ satisfies condition $(I_2)$}. \\

Define $z:= \frac{e_{n+1}}{\|e_{n+1}\|}r$, where $r$ is given in $(\ref{eq 22})$. Let $\rho_0>0$. Consider the map (for $m= d+1$ where $d:= \dim Y$)
\begin{eqnarray*}
	\Psi : \mathbb{R}^m  &\to& \Psi(\mathbb{R}^m) \subset \mathbf{\mathit{graph}\,F}\\
	\mu=(\mu_1,\ldots,\mu_m)
	&\mapsto& \begin{cases}
		\Bigg(\rho_0\big(\sum_{i=1}^d \mu_i e_i + \mu_m z\big), \;I_{g^*}\Big(\rho_0(\mu_1 e_1+ \cdots+\mu_de_d+ \mu_mz)\Big) \Bigg), &\text{if}\;\; \mu_m\ge0,\\[1.2 em]
		\Bigg(\rho_0\sum_{i=1}^d \mu_i e_i, \;I_{g^*}\Big(\rho_0(\mu_1 e_1+ \cdots+\mu_de_d)\Big)\Bigg), &\text{if}\;\; \mu_m\le0.
	\end{cases}
\end{eqnarray*}and
	\begin{equation*}
		g^*(x)= \begin{cases}
			\overline{g}(x)\quad \text{if}\quad x\ge0,\\
			\underline{g}(x)\quad\text{if}\quad x<0.
		\end{cases}
	\end{equation*}
	Set
	\begin{equation*}
	L(\mu_1,\ldots,\mu_m)
	:= 
	\begin{cases}
		\sum_{i=1}^d \mu_i e_i + \mu_m z, &\text{if}\;\; \mu_m\ge0,\\[0.5 em]
		\sum_{i=1}^d \mu_i e_i, &\text{if}\;\; \mu_m\le0.
	\end{cases}
		\end{equation*}

The map $\Psi$ is a homeomorphism: Since the functional $I_{g^*}$ is continuous, $\Psi$ is continuous.

	--	\textbf{Surjective.} By construction, $\Psi$ is surjective.

	--	\textbf{Injective.}
	Suppose that $L(\mu)=0$. Then, 
	$\sum_{i=1}^d \mu_i e_i =-\mu_m z$. Since $X=Y\oplus Z$, we have  $\mu_m=0$, and thus $\mu_1=\cdots=\mu_d=0$ because $(e_i)_{1\le i\le d}$ is a basis of $Y$. Therefore \(\Psi\) is injective.
	  
	-- \textbf{Inverse map.} For $(u, c)\in \Psi(\mathbb{R}^{m})$, the inverse map $\Psi^{-1}$ of $\Psi$ is 
	\begin{equation*}
		\Psi^{-1}(u, c) =[u]_{\mathcal{B}_m},
	\end{equation*}
	where $[u]_{\mathcal{B}_m}$ is the representation of $u$ in the basis $\mathcal{B}_m=\{e_1, \cdots, e_{d}, z\}$. Since the mapping $u\mapsto [u]_{\mathcal{B}_m}$ is continuous, it follows that $\Psi^{-1}$ is continuous.

	First, observe that, by assumption $(G_3)$(2) and Remark $\ref{rm 3}$, on $Y$, we have 
	\[I_{g^*}(u) \le \int_{\Omega}\Big( \lambda_{n}\frac{u^2(t)}{2}-\int_{0}^{u(t)}g(x)dx\Big)dt\le 0.\]
	
	For every $w\in E$ and every $g\in \mathcal{S}(G)$, relation $(\ref{eq 19})$ together with the Hölder inequality yield
		\begin{equation*}
		I_g(w)\le \dfrac{\|w\|^2}{2} + \|\varphi\|_{L^{N/2}(\Omega)} \frac{\|w\|^2_{L^{2^*}(\Omega)}}{2} -k \|w\|^{\beta}_{L^\beta(\Omega)}+l\,|\Omega|,\;\; \text{where}\;\; 2^*:= \dfrac{2N}{N-2}.
	\end{equation*}
	
	Since, on the finite dimensional space $Y\oplus \mathbb{R}.z$, all norms are equivalent, for every $w\in Y\oplus \mathbb{R}.z$ and every $g\in \mathcal{S}(G)$, we have
	\begin{equation*}
		I_g(w)\le \dfrac{\|w\|^2}{2} + \|\varphi\|_{L^{N/2}(\Omega)} \frac{\|w\|^2}{2} -k \|w\|^{\beta}+l\,|\Omega|.
	\end{equation*}
	
	Since $\beta>2$, then
	\begin{equation*}
		I_g(w)\to -\infty, \;\; \text{whenever}\;\; \|w\| \to + \infty,\;\; w\in Y\oplus \mathbb{R}.z.
	\end{equation*}
	Thus, we can find $\rho>r>0$, $\rho$ sufficiently large such that
	\begin{equation*}
	\underset{w\in \partial M}{\max}\,I_g(w)=0, 
	\end{equation*}
	where $\partial M$ is defined in Theorem $\ref{theo 4.2}$.
	
	Now, define the map
\begin{eqnarray*}
	V : \partial B_1^m   &\to& Y\oplus\mathbb{R}.z\\
	\mu= (\mu_1, \ldots, \mu_m) &\mapsto& \begin{cases}
		\dfrac{\rho L(\mu)}{\|L(\mu)\|},\;\; &\text{if}\;\; \mu_m\ge0,\\[1 em]
		\min\,\left\{ 1, 	\dfrac{\rho}{\|L(\mu)\|} \right\}L(\mu), \;\; &\text{if}\;\; \mu_m\le0.
	\end{cases}
\end{eqnarray*}
	By definition, $V(\mu)\subseteq \partial M$. Therefore, we can choose $\rho_0$ in such a way that $V(\mu)$ represents the component of $\Psi_m(\partial B_1^m)$ in $E$. That is,
	\[\Psi_m(\partial B_1^{m})\subseteq \partial M \times (-\infty, 0].\]
	\end{proof}
	\begin{proof}[\bf{Proof of Theorem $\ref{theo 5.1}$ $($Version 1$)$}] By Lemma $\ref{lem 5.8}$, the conditions $(I_1)$ and $(I_2)$ in Theorem $\ref{theo 4.2}$ are verified. 
	By Theorem $\ref{theo 4.2}$, there exists a sequence $(u_n,c_n)\subset \mathbf{\mathit{graph}\,F}$ satisfying
	\begin{equation}
		c_n\to c, \;\; |dF|(u_n,c_n)\to 0,\;\; n\to \infty.
	\end{equation}
	By Lemma $\ref{lem 5.7}$, the sequence $(u_n)$ is bounded. 
	Therefore, by Rellich embedding theorem, $(u_n)$ is relatively compact in $L^{\mu+1}(\Omega)$ because $2<\mu+1< \frac{2N}{N-2}$. There exists a subsequence of $(u_n)$ still denoted $(u_n)$ that converges in $L^{\mu+1}(\Omega)$. \\
	There exists $\mathscr{C}>0$ such that $\|h_n\|_{L^q(\Omega)}\le \mathscr{C}$, where $\frac{1}{\mu +1}+\frac{1}{q}=1$ for every $n$. 
	Indeed, by assumption $(G_1)$, we have 
	\begin{equation}
		\label{eq 24}
		\int_{\Omega} |h_n(t)|^q dt \le C_1 \int_{\Omega} |u_n(t)|^{\mu q} dt + C_2 \int_{\Omega}dt = C_1 \int_{\Omega} |u_n(t)|^{\mu +1} dt + C_2 \int_{\Omega}dt \le \mathscr{C}<\infty,
	\end{equation}
	for some positives constants $C_1$ and $C_2$ and where we use the following standard inequality: 
	if $a_1,\dots,a_n \geq 0$ and $\kappa \ge 1$, then 
	\[
	\Big(\sum_{i=1}^n a_i\Big)^\kappa \leq n^{\,\kappa-1}\sum_{i=1}^n a_i^\kappa.
	\]
	
	We would like to show that $(u_n)$ converges in $E$. To this end, we show that $(u_n)$ is a Cauchy sequence. Since $L^p(\Omega)$, $1\le p<\infty$ is complete space, we have \[\|u_n -u_m\|_{L^p(\Omega)} \to 0,\;\; \text{whenever}\;\; n,\; m\to \infty.\]
	Using Hölder inequality, we have
	\[
	\begin{aligned}
		\|u_n - u_m\|^2
		&= \int_{\Omega} \Bigg(\nabla\Big(u_n(t) - u_m(t)\Big)\cdot \nabla\Big(u_n(t) - u_m(t)\Big)\Bigg)\, dt \\[4pt]
		&= \int_{\Omega} \Bigg(\nabla u_n(t) \cdot \nabla\Big(u_n(t) - u_m(t)\Big)
		+ \nabla u_m(t) \cdot \nabla\Big(u_m(t) - u_n(t)\Big)\Bigg)\, dt \\[4pt]
		&= \langle \alpha_n , u_n - u_m \rangle
		- \int_{\Omega} \varphi(t) u_n(t)\Big(u_n(t) - u_m(t)\Big) dt
		+ \int_{\Omega} h_n(t) \Big(u_n(t) - u_m(t)\Big)\, dt \\[4pt]
		&\qquad \qquad
		+ \langle \alpha_m , u_m - u_n \rangle - \int_{\Omega} \varphi(t) u_m(t)\Big(u_m(t) - u_n(t)\Big) dt
		+ \int_{\Omega} h_m(t) \Big(u_m(t) - u_n(t)\Big)\, dt \\[4pt]
		&\le \langle \alpha_n , u_n - u_m \rangle
		+ \langle \alpha_m , u_m - u_n \rangle - \int_{\Omega} \varphi(t) \Big(u_n(t) - u_m(t)\Big)^2 dt\\[4pt]
		&\qquad\qquad\qquad\qquad\qquad
		+ \int_{\Omega} \Big|h_n(t)\Big(u_n(t) - u_m(t)\Big)\Big|\, dt
		+ \int_{\Omega} \Big|h_m(t)\Big(u_m(t) - u_n(t)\Big)\Big|\, dt \\[4pt]
		&\le \|\alpha_n\|_* \|u_n - u_m\|
		+ \|\alpha_m\|_* \|u_n - u_m\|
		+ \mathscr{C} \|u_n - u_m\|_{L^{\mu+1}(\Omega)}
		+ \mathscr{C} \|u_m - u_n\|_{L^{\mu+1}(\Omega)} \\[4pt]
		&\qquad\qquad\qquad\qquad\qquad\qquad 
		- \int_{\Omega} \varphi(t) \Big(u_n(t) - u_m(t)\Big)^2 dt\\
		&= C\bigl(\|\alpha_n\|_* + \|\alpha_m\|_*\bigr)
		+ 2\mathscr{C} \|u_n - u_m\|_{L^{\mu+1}(\Omega)} -\int_{\Omega} \varphi(t) \Big(u_n(t) - u_m(t)\Big)^2 dt,
	\end{aligned}
	\] for some positive constant $C$.
	
	From Lemma $\ref{lem 5.1}$ and the fact that $\|\alpha_n\|*\le |dF|(u_n, c_n)$ and $\|\alpha_m\|*\le |dF|(u_m, c_m)$, we get
	\[\|\alpha_n\|_* \to 0,\; \|\alpha_m\|_*\to 0, \;\int_{\Omega} \varphi(t) \Big(u_n(t) - u_m(t)\Big)^2 dt \to 0, \;\; n, \; m\to \infty.\]
	Thus we have proved that 
	\begin{equation*}
		\|u_n-u_m\|\to 0,\;\; n,\; m\to \infty.
	\end{equation*}
	The sequence $(u_n)$ is then a Cauchy sequence in the complete space $E$, and therefore it converges in $E$.

	The Theorem $\ref{theo 3.3}$ gives the existence of a critical point $u_0$ of $F$ at level $c$. Lemma $\ref{lem 5.6}$ implies that $u_0$ is a solution of problem $(\ref{eq 11})$. 
	\end{proof}
	\begin{proof}[\bf{Proof of Theorem $\ref{theo 5.1}$ $($Version 2$)$}] By Lemma $\ref{lem 5.8}$, the conditions $(I_1)$ and $(I_2)$ in Theorem $\ref{theo 4.2}$ are verified. 
		By Theorem $\ref{theo 4.2}$, there exists a sequence $(u_n,c_n)\subset \mathbf{\mathit{graph}\,F}$ satisfying
		\begin{equation}
			c_n\to c, \;\; |dF|(u_n,c_n)\to 0,\;\; n\to \infty.
		\end{equation}
		By Lemma $\ref{lem 5.7}$, the sequence $(u_n)$ is bounded. 
	Therefore, by Rellich embedding theorem, $(u_n)$ is relatively compact in $L^{\mu+1}(\Omega)$ because $2<\mu+1< \frac{2N}{N-2}$.	The sequence \((u_n)\) admits a convergent subsequence in \(L^{\mu+1}(\Omega)\), which we still denote by \((u_n)\). We also have \(u_n \to u\) in \(L^{\mu +1}_{\mathrm{loc}}(\Omega)\).  
		For every \(v \in \mathcal{C}_0^\infty(\Omega)\), since \(\|\alpha_n\|_* \le |dF|(u_n,c_n) \to 0\) as \(n \to \infty\), we deduce that  
		\[
		\langle \alpha, v \rangle = \lim_{n\to\infty} \langle \alpha_n, v \rangle = 0.
		\]
		Hence \(\alpha = 0\), and \(u\) is a nontrivial weak solution in \(E\) of problem \((P)\).
		
		Indeed, let \(K \subset \Omega\) be a compact set such that \(\mathrm{supp}\, v \subset K\). We have
		\[
		\langle \alpha_n , v \rangle
		= \int_{K} \Big( \nabla u_n(t)\cdot\nabla v(t) + \varphi(t) u_n(t) v(t) \Big)\, dt
		- \int_{K} h_n(t) v(t)\, dt.
		\]
		
		By \((\ref{eq 24})\), the sequence \((h_n)\) is bounded in \(L^\tau(K)\), with \(1 < \tau < +\infty\). Therefore, there exist \(h \in L^\tau(K)\) and a subsequence of \((h_n)\), still denoted \((h_n)\), such that  
		\[
		h_n \rightharpoonup h \;\; \text{in } L^\tau(K), \;\; 1 < \tau < +\infty.
		\]
		As in the proof of Lemma \(\ref{lem 5.4}\), one can show that \(h \in \mathcal{S}(u)\). It follows that, for every \(v \in \mathcal{C}_0^\infty(\Omega)\),
		\[
		\int_K h_n(t) v(t)\, dt \to \int_K h(t) v(t)\, dt, \;\; n \to \infty.
		\]
		
		On the other hand, using Hölder’s inequality and the fact that \(v\) is bounded, there exists \(c>0\) such that
		\[
		\int_K \Big| \varphi(t)\big(u_n(t)-u(t)\big) v(t) \Big|\, dt
		\le c \|\varphi\|_{L^{N/2}(K)} \|u_n - u\|_{L^{\mu +1}(K)} \to 0, \;\; n \to \infty,
		\]
		where
		\[
		\frac{2}{N} + \frac{1}{\mu +1} = 1.
		\]
		
		We deduce that
		\begin{multline*}
			\langle \alpha_n , v \rangle - \langle \alpha , v \rangle
			= \int_{K} \Big( \nabla (u_n(t)-u(t)) \cdot \nabla v(t)
			+ \varphi(t) (u_n(t)-u(t)) v(t) \Big)\, dt \\
			- \int_{K} (h_n(t) - h(t)) v(t)\, dt
			\rightarrow 0, \;\; n \to \infty.
		\end{multline*}
		
		\end{proof}

\begin{remark}
	If $\lambda_1>0$, it suffices to use The Mountain Pass Theorem \cite[Theorem 2.14]{Fri} instead of the Linking Theorem $\ref{theo 4.2}$.
\end{remark}

\end{document}